\documentclass{aims}

\usepackage{amsmath}
\usepackage{paralist}

\def\currentvolume{X}
\def\currentissue{X}
\def\currentyear{200X}
\def\currentmonth{XX}
\def\ppages{X--XX}

\usepackage{hyperref}
\usepackage[T1]{fontenc}
\usepackage{mathrsfs}
\usepackage{amsmath,amsthm,amssymb,psfrag}
\usepackage{pxfonts}
\usepackage{latexsym}
\usepackage{times}
\usepackage{enumerate}
\usepackage{mathrsfs}
\usepackage{stmaryrd}
\usepackage{amsopn}
\usepackage{amsmath}
\usepackage{amssymb}
\usepackage{amsfonts}
\usepackage{amsbsy}
\usepackage{amscd,indentfirst,epsfig}
\usepackage{hyperref}
\usepackage{amsfonts,amsmath,latexsym,amssymb,verbatim,amsbsy}
\usepackage{amsthm}
\usepackage{dsfont}
\usepackage{colordvi}
\usepackage{pstricks}
\usepackage{graphicx}

\def\thesection{\arabic{section}}
\renewcommand{\theequation}{\arabic{section}.\arabic{equation}}

\newtheorem{theo}{Theorem}[section]

\newtheorem*{lemmeA}{Lemma A.1}
\newtheorem*{nbA}{Remark A.1}

\newtheorem{propo}[theo]{Proposition}
\newtheorem{cor}[theo]{Corollary}

\newtheorem{nb}[theo]{Remark}

\theoremstyle{definition}

\def \leq {\leqslant}
\def \geq {\geqslant}
\def \le {\leq}
\def \ge  {\geq}
\numberwithin{equation}{section}
\def\ind#1{\lower5pt\hbox{$\scriptstyle #1$}}

\def \x {x}
\def \l {\ell}
\def \d {\,\mathrm{d} }

\def \e {e}

\def \ds {\displaystyle}
\def \l {\ell}
\def \a {\alpha}
\def \b {\beta}
\def \ap {\mu}
\def \D {\mathscr{D}}

\def \ds {\displaystyle}
\def\Q {\mathcal{Q}}

\def\R{{\mathbb R}}
\def \S {{\mathbb S}^2}
\def \q {q}
\def \n {n}
\def \qn {|\q \cdot \n|}
\def \u {{u}_1}
\def \v {{v}}
\def \vh {\v^{\star}}
\def \vb {\v_{\star}}
\def \w {{w}}
\def \wh {\w^{\star}}
\def \wb {\w_{\star}}

\def \It {\int_{\R^3 \times \S}}

\def \var {\varepsilon}

\def\cn{\cdot \n}
\def \O {\mathbf{\Omega}}
\title[Inelastic scattering Boltzmann models]
{Relaxation rate, diffusion approximation and Fick's law for inelastic
scattering Boltzmann models}

\author[Bertrand Lods, Cl\'ement Mouhot and Giuseppe Toscani]{}


\email{bertrand.lods@math.univ-bpclermont.fr}
 \email{clement.mouhot@ceremade.dauphine.fr}
 \email{giuseppe.toscani@unipv.it}

\thanks{G.T. acknowledges the support of the Italian MIUR project
  ``Kinetic and hydrodynamic equations of complex collisional systems''.}

\begin{document}
\maketitle
\centerline{\scshape Bertrand Lods }
\medskip
{\footnotesize
 \centerline{Laboratoire de Math\'{e}matiques, CNRS UMR 6620}
   \centerline{Universit\'{e} Blaise Pascal (Clermont-Ferrand 2), 63177 Aubi\`{e}re Cedex,
France.}
} 

\centerline{\scshape Cl\'{e}ment Mouhot }
\medskip
{\footnotesize
 \centerline{{CNRS} \& Universit\'{e} Paris-Dauphine }
   \centerline{UMR7534, F-75016 Paris, France}
} %

\centerline{\scshape Giuseppe Toscani }
\medskip
{\footnotesize
  \centerline{Department of Mathematics at the University of Pavia}
  \centerline{via Ferrata 1, 27100 Pavia, Italy.}
}

\bigskip

 \centerline{(Communicated by the associate editor name)}

\hyphenation{bounda-ry rea-so-na-ble be-ha-vior pro-per-ties cha-rac-te-ris-tic
coer-ci-vity}


 \begin{abstract}
We consider the linear dissipative Boltzmann equation describing
inelastic interactions of particles with a fixed background. For the
simplified model of Maxwell molecules first, we give a complete
spectral analysis, and deduce from it the optimal rate of
exponential convergence to equilibrium. Moreover we show the
convergence to the heat equation in the diffusive limit and compute
explicitely the diffusivity. Then for the physical model of hard spheres
we use a suitable entropy functional for which we prove explicit
inequality between the relative entropy and the production of
entropy to get exponential convergence to equilibrium with explicit
rate. The proof is based on inequalities between the entropy
production functional for hard spheres and Maxwell molecules.
Mathematical proof of the convergence to some heat equation in the
diffusive limit is also given. From the last two points
we deduce the first explicit estimates on the diffusive coefficient
in the Fick's law  for (inelastic hard-spheres) dissipative gases.
 \end{abstract}



\section{Introduction}
\setcounter{equation}{0}

The  linear Boltzmann equation for   {\em granular particles} models
the dynamics of dilute particles (test particles with negligible
mutual interactions) immersed in a fluid at thermal equilibrium that
undergo {\em inelastic collisions} characterized by the fact that
the total kinetic energy of the system
 is dissipated during collision. Such an equation introduced in \cite{MaPi, SpTo, LoTo}
provides an efficient description of the dynamics of a mixture of
impurities in a gas \cite{FePa,BrPa}. Assuming the fluid at thermal
equilibrium and neglecting the mutual interactions of the particles,
the evolution of the distribution of the particles
phase is modelled by the linear Boltzmann equation which reads
\begin{equation} \label{evol}
  {\partial_t f}  + \v \cdot \nabla_{\x} f = \Q(f),
 \end{equation} with suitable initial condition $f(\x,\v,0)=f_0(\x,\v)$,
 $(x,v) \in \R_\x^3 \times \R_\v^3$. Here above, the collision
operator $\Q(\cdot)$ is a \textit{linear scattering operator} given
by $$\Q(f)=\mathcal{B}(f,\mathcal{M}_1),$$ where
$\mathcal{B}(\cdot,\cdot)$ denotes the usual quadratic Boltzmann
collision operator for granular gases (cf. \cite{BrPo} for instance)
and $\mathcal{M}_1$ stands for the distribution function of the host
fluid which is assumed to be a given Maxwellian with bulk velocity
$u_1$ and temperature $\Theta_1$ (see Section 2 for details). Notice
that we shall deal in this paper with the collision operator $\Q$
corresponding to \textbf{\textit{hard-spheres interactions}} as well
as with the one associated to  \textbf{\textit{Maxwell molecules
interactions}}. The inelasticity is modeled by a {\bf
\textit{constant normal restitution coefficient}}. The main goals
and results of this paper are the following:

\begin{enumerate}[(1)\;]
\item First, we explicit the exponential rate of convergence towards equilibrium for  the solution
to the space-homogeneous version of \eqref{evol} (for both
Maxwellian molecules and hard-spheres interactions) \textbf{\textit{through a
quantitative estimate of the spectral gap of $\Q$}}. It is computed in an
explicit way for Maxwell molecules, and estimated in an explicit way
for hard-spheres, see Theorem~\ref{coercivity} (together with its
Corollary~\ref{dechomogene} for its consequence on the asymptotic
behavior of space-homogeneous solutions), where $\mu_{\mathsf{max}}$ is defined
in Theorem~\ref{Spec}, and the constant $C^*$ is detailed in
Remark~\ref{nb:lower}.

\item Second, we investigate the problem of the \textbf{\textit{diffusion
approximation for \eqref{evol}}}. Precisely, we show that the
macroscopic limit $\varrho$ of Eq. \eqref{evol} in the diffusive
scaling  is the solution to some (parabolic) heat equation (see
Proposition~\ref{weakconv} together with Theorems \ref{limdiff} and
\ref{FickHS}). \textbf{\textit{When
dealing with Maxwell molecules interactions, the diffusivity of this
heat equation can then be computed explicitly}} (see Theorem~\ref{limdiff}
together with the computation of Remark~\ref{nb:FickMax} for the
diffusivity). \textbf{\textit{This is no more the
case for the equation corresponding to hard-spheres interactions
but we provide some new quantitative estimates on it}} (see Theorem~\ref{FickHS}
together with the estimate of Proposition~\ref{FickHSestim}).
\end{enumerate}

Concerning point (1), it is known from \cite{MaPi,SpTo,LoTo} that the linear
collision operator admits a unique steady state given by a (normalized)
Maxwellian distribution function $\mathcal{M}$ with bulk velocity $u_1$ and
temperature $\Theta^{\#} \leq \Theta_1$. Moreover, thanks to the spectral
analysis of $\Q$ performed in \cite{ArLo} (in the hard-spheres case), the
solution to the space-homogeneous version of \eqref{evol} is known to converge
exponentially (in some pertinent $L^2$ norm) towards this equilibrium
$\mathcal{M}$ as times goes to infinity. This exponential convergence result is
based upon the existence of a positive spectral gap for the collision operator
$\Q$ and relies on compactness arguments, via Weyl's Theorem.
\textbf{\textit{Because of this non constructive approach, at least for
hard-spheres interactions, no explicit estimate on the relaxation rate were
available by now}}. It is one of the objectives of this paper to fill this
blank. It is well-known that the kinetic description of gases  through the
Boltzmann equation is relevant \textit{only on some suitable time scale}
\cite{BrPo,Ce88}. Providing explicit estimates of the relaxation rate is the
only way to make sure that the  time scale for the equilibration process is
smaller than  the one on which the kinetic modeling is relevant. Another
motivation to look for an explicit relaxation rate relies more on
methodological aspects. Compactness methods do not rely on any physical
argument and it seems to us more natural to look for a method which relies as
much as possible on physical mechanisms, e.g. dissipation of entropy.

Precisely, the strategy we adopt to treat the  above point (1) is
based upon an explicit estimate of the spectral gap of $\Q$. For
Maxwell molecules interactions, we use the Fourier-based approach
introduced by Bobylev \cite{Bo88} for the study of the linearized
(elastic) Boltzmann equation and then used for the study of the
spectrum of the {\em linearized} inelastic collision operator in
\cite{BCG}, and we provide an explicit description of the whole
spectrum  of  this linear {\em scattering} collision operator .
Then, to treat the case of hard-spheres interactions, our method is
based upon the \textit{entropy-entropy-production method}.
Precisely, we show that the entropy production functional (naturally
associated to the $L^2(\mathcal{M}^{-1})$ norm) corresponding to the
hard-spheres model can be bounded from below (up to some explicit
constant) by the  one associated to the Maxwell molecules model
(Proposition \ref{linkHM}). Note that such a comparison between
entropy production rates for hard-spheres and Maxwell molecules
interactions is inspired by the approach of \cite{BaMo} which deals
with the linearized (elastic) Boltzmann equation. In the present
case, the method of proof is different and simpler, being based upon
the careful study of a convolution integral. Such an entropy
production estimate allows us to prove, via a suitable coercivity
estimate of $\Q$ (Theorem \ref{coercivity}), that any
space-homogenous solution to \eqref{evol} converges exponentially
towards equilibrium with an explicit rate that depends on the model
under investigation.\\

Concerning now  point (2), various attempts to derive hydrodynamic equations
from the dissipative nonlinear Boltzmann equation exist in the literature,
mostly based upon suitable moment closure methods \cite{To,BiSp,BiSpTo} or on
the study of the linearized version of the Boltzmann equation around
self-similar solutions (homogeneous cooling state) \cite{BaDuBr,BrDu} in some
weak inelastic regime. Dealing with the linear Boltzmann equation \eqref{evol},
hydrodynamic models describing the evolution of the momentum and the
temperature of the gas have been obtained in \cite{BrPa} as a closed set of
\textit{dissipative Euler equations} for some pseudo-Maxwellian approximation
of $\Q$. Similar results have been obtained in \cite{FePa} where numerical
methods are proposed for the resolution of both the kinetic and hydrodynamic
models. The work \cite{BiSp} proposes two closure methods, based upon a maximum
entropy principle, of the moment equations for the density, macroscopic
velocity and temperature. These closure methods lead to a single diffusion
equation for the hydrodynamical variable. In the present paper, we shall
discuss the diffusion approximation of the linear Boltzmann equation
\eqref{evol} with the main objective of providing a \textbf{\textit{rigorous
derivation of the Fick's law for dissipative gases and an estimate on the
diffusive coefficient}}. Recall that the diffusion approximation for the linear
Boltzmann equation consists in looking for the limit, as the small parameter
$\var$ goes to $0$, of the solution to the following re-scaled kinetic
equation:
\begin{equation}\label{rescaled} \varepsilon  \partial_t
f_\varepsilon(t,\x,\v)+\v \cdot \nabla_\x
f_\varepsilon(t,\x,\v)=\dfrac{1}{\varepsilon}  \Q(f_\varepsilon)(t,\x,\v),\end{equation}
with suitable initial condition. We consider indeed here the
\textit{Navier-Stokes scaling}, namely, we assume the mean free path
to be a small parameter $\lambda=\var \ll1 $ and, at the same time,
we rescale time as $t \to t/\var$ in order to see emerging the diffusive
hydrodynamical regime (and not the Euler hydrodynamical description, which would
be a trivial transport equation in our case). Performing
a formal Hilbert asymptotic expansion of the
solution 
allows us to expect the solution $f_\var$ to converges towards a
limit $f $ with $\Q(f)=0.$ Therefore, the expected limit of $f_\var$
is of the form $f(t,x,v)=\varrho(t,x)\mathcal{M}(v),$ and the
diffusion approximation problem consists in expressing
$\varrho(t,x)$ as the solution to some suitable diffusion equation.
Actually, standard approach consists in using the continuity
relation $$\partial_t \varrho(t,x)+\mathrm{div}_\x j(t,x)=0,$$
between the density $\varrho$ and the current vector $j(t,x)$
together with a suitable {\textit{Fick's law}} that links the
current $j$ to the gradient of $\varrho$: $$j(t,x)= -\mathsf{D}
\, \nabla_\x \varrho(t,x)$$ for some suitable diffusion coefficient
(diffusivity) $\mathsf{D}
>0$ which depends on the kind of interactions we are dealing with.
For Maxwell molecules interactions, the expression of the
diffusivity can be made explicit while this is no more the case when
dealing with hard-spheres model. The method we adopt for the proof
of the diffusive limit follows very closely  the work of P. Degond,
T. Goudon and F. Poupaud \cite{DeGoPo}. Though more general than
ours since it deals with models without detailed balance law, the
study of \cite{DeGoPo} is restricted to the case of a collision
operator for which the collision frequency is controlled from above
in a way that excludes the case of physical hard-spheres
interactions (recall that, for hard-spheres, the collision frequency
behaves asymptotically like $(1+|\v|)$ \cite{ArLo}). Actually, the
analysis of \cite{DeGoPo} can be make valid under the only
hypothesis that the coercivity estimate obtained in Theorem
\ref{coercivity} (and strenghten in Theorem \ref{coercivity2}) holds
true. Precisely, such a coercivity estimate of $\Q$ allows to obtain
satisfactory {\it a priori bounds}
 for the solution to the re-scaled equation
\eqref{rescaled}. We are then able to prove the weak convergence of
the density $\varrho_\var$ and current $j_\var$ of the solution
$f_\var$ towards suitable limit density $\varrho$ and limit current
$j$. It is also possible, via compensated-compactness arguments from
\cite{LiTo, Mu, Ta} to prove strong convergence result in $L^2$
norm.

The organization of the paper is as follows. In Section
\ref{sec:preli} we present the models we shall deal with as well as
some related  known results we shall use in the sequel. In
Section \ref{sec:entropy} we perform the computations of the
spectrum in the Maxwell molecules case and then prove the crucial
entropy production estimates in the hard-spheres case (from which we deduces the
explicit convergence rate to equilibrium for the space-homogenous
version of \eqref{evol}). Section \ref{sec:diff} is dealing with the
above point $(2)$. We first prove {\it a priori}
estimates valid for both models of hard-spheres and Maxwell
molecules and based upon the coercivity estimates obtained in
Section \ref{sec:entropy}. Then, we deal separately with the cases
of Maxwell molecules and hard-spheres proving for both models the
convergence towards suitable macroscopic equations, providing for
Maxwell molecules an explicit expression of the diffusivity, and
for hard-spheres explicit estimates on it.

\section{Preliminaries}\label{sec:preli}
\setcounter{equation}{0}

\subsection{The model} As explained in the introduction, we shall deal with a \textit{linear scattering operator} $\Q$ given
by
\begin{equation}\label{linearoperator}
\Q (f)=\dfrac{1}{2\pi \lambda}\It B(\q, \n) \left[J_B
f(\x,\vb,t)\mathcal{M}_1(\wb)
-f(\x,\v,t)\mathcal{M}_1(\w)\right]\d\w \d\n\end{equation} where
$\lambda$ is the mean free path, $\q=\v-\w$ is the relative
velocity, $\v_\star$ and $\w_\star$ are the pre-collisional
 velocities   which result, respectively, in $ \v $ and $\w$ after
 collision.  The main
feature of the \textit{binary dissipative
 collisions} is that part of the normal relative
velocity is lost in the interaction, so that
\begin{equation}\label{coef}
(\vh-\wh) \cn = - \e(\v-\w)\cn ,
\end{equation}
where $\n \in \S$ is the unit vector in the direction of impact and
$0 < \e < 1$ is the so-called normal restitution coefficient.
Generally, such a coefficient should depend on $(\v,\w)$ but, for
simplicity, we shall only deal with a \textit{\textbf{constant
normal restitution coefficient}} $\e$. The collision kernel
$B(\q,n)$ depends on the microscopic interaction (see below) while
the term $J_B$ corresponds to the product of the Jacobian of the
transformation $(v_\star,w_\star) \to (v, w)$ with the ratio of the
lengths of the collision cylinders \cite{BrPo}. Note that in such a
scattering model, the microscopic masses of the dilute particles $m$
and that of the host particles $m_1$ can be different. We will
assume throughout this paper that the distribution function
$\mathcal{M}_1$ of the host fluid is a given normalized Maxwellian
function:
\begin{equation}\label{maxwe1}
\mathcal{M}_1(\v)=\bigg(\dfrac{m_1}{2\pi \Theta_1}\bigg)^{3/2}\exp
\left\{-\dfrac{m_1(\v-u_1)^2}{2\Theta_1}\right\}, \qquad \qquad \v
\in \R^3,
\end{equation}
where $u_1 \in \R^3$ is the given bulk velocity  and $\Theta_1
>0$ is the given effective temperature of the host fluid. For particles of masses $m$ colliding inelastically with particles
of mass $m_1$, the restitution coefficient being constant, the
expression of the pre-collisional velocities $(\v_\star,\w_\star)$
are given by  \cite{BrPo,SpTo}
 \begin{equation}
\v_{\star}=\v-2\alpha\dfrac{1-\beta}{1-2\beta} \left(\q \cdot
\n\right) \n,\qquad
\w_{\star}=\w+2(1-\alpha)\dfrac{1-\beta}{1-2\beta} \left(\q \cdot
\n\right)  \n;
\end{equation}
where $\q=\v-\w$, $\alpha$ is the mass ratio and $\beta$ denotes the
inelasticity parameter
  \[ \alpha = \frac{m_1}{m+m_1} \in (0,1), \ \ \ \beta=\frac{1-\e}{2} \in [0,1/2). \]
We shall investigate in this paper several collision operators
corresponding to various interactions collision kernels. Namely, we
will deal with
\begin{itemize}
\item the linear Boltzmann operator for
\textbf{\textit{hard-sphere interactions}} $\Q=\Q_{\mathsf{hs}}$ for
which
$$B(\q,n)=B_{\mathrm{hs}}(\q,n)=|\q \cdot \n|, \qquad \text{ and } \qquad J_{B_{\mathrm{hs}}}=:J_\mathsf{hs}=\dfrac{1}{\e^2};$$
\item the scattering operator $\Q=\Q_{\mathsf{max}}$ corresponding
to the \textbf{\textit{Maxwell molecules approximation}} for which
$$B(\q,n)=B_{\mathsf{max}}(q,n)=|\widetilde{\q} \cdot \n|, \quad \widetilde{q}=\q/|\q|, \qquad
\text{ and }  \qquad
J_{B_{\mathsf{max}}}=J_\mathsf{max}=\dfrac{1}{\e^2}\dfrac{|v-w|}{|v^\star-w^\star|}.$$
\end{itemize}
It will be sometimes convenient to express the collision operator
$\Q$ in the following weak form:
\begin{equation}\label{weak}
  <  \psi, \Q(f) > =\frac{1}{2 \pi \lambda} \,
       \int _{\R^3 \times \R^3 \times \S} B(\q,n)
       f(\v)\mathcal{M}_1(\w)\left[\psi(\v^\star)-\psi(\v)\right]\d
       v \d w \d \n \end{equation}
for any regular $\psi$, where  $(v^\star,\w^\star)$ denote  the
post-collisional velocities  given by
  \begin{equation}
  \v^\star = \v - 2\alpha (1-\beta)\left(\q \cdot
\n\right)
  \n,\qquad \w^\star = \w + 2(1-\alpha) (1-\beta)\left(\q \cdot
\n\right)  \n.
  \end{equation}In particular, one sees that the dissipative feature of microscopic collision is
measured, at the macroscopic level, only through the parameter:
$$\kappa=\a(1-\b)=\dfrac{\a}{2}(1+\e) \in (0,1)$$
appearing in the expression of $\v^\star$. Accordingly, the
macroscopic properties of $\Q$ are those of the classical linear
Boltzmann gas whenever $\kappa=1/2$ (which is equivalent to
$v^\star=\v$).
It can be shown for both cases that the number
density  of the dilute gas is the unique conserved macroscopic
quantity (as in the elastic case). In contrast with the nonlinear
Boltzmann equation for granular gases, the temperature, though  not
conserved, remains bounded away from zero, which prevents the
solution to the linear Boltzmann equation to converge towards a
Dirac mass.

Moreover let us remark that from the dual form we see that the
collision operator in fact depends only on two real parameters $m_1$
and $\kappa$ (plus $u_1$ of course) and not $u_1$ plus three
parameters $\alpha, \beta, m_1$ as a first guess would suggest.

\subsection{Universal equilibrium and $H$-Theorem}
A very important feature of these inelastic scattering models is the
existence (and uniqueness) of a {\em universal equilibrium}, that is
independent of range of the microscopic-interactions (that is of the collision
kernel $B$) and depending only on the parameters $m_1$, $\kappa$ and
$u_1$. Precisely, the background forces the system to adopt a
Maxellian steady state (with density equal to $1$):
\begin{theo}\label{steadystate} The  Maxwellian velocity distribution:
\begin{equation}\mathcal{M}(\v)=\left(\dfrac{m}{2\pi
\Theta^{\#}}\right)^{3/2}\exp
\left\{-\dfrac{m(\v-\u)^2}{2\Theta^{\#}}\right\}, \qquad \v \in
\R^3,\end{equation} with
\begin{equation}\Theta^{\#}
=\dfrac{(1-\a)(1-\b)}{1-\a(1-\b)}\Theta_1\end{equation} is the
unique equilibrium state of $\Q$ with unit mass.
\end{theo}

Note that this universal equilibrium is coherent with the remark that the
collision only depends on $m_1$ and $\kappa$ (and $u_1$) from the dual form, since
$$
\frac{m}{\Theta^{\#}} = \frac{m_1}{\Theta_1}\frac{1-
\kappa}{\kappa}.
$$

This explicit Maxwellian equilibrium state  allows to develop
 entropy, spectral and hydrodynamical analysis on both the models in the same
way. First, from~\cite{Pe93,Pe03} and~\cite{LoTo}, the existence and uniqueness
of such an equilibrium state allows to establish a linear version of
the famous $H$--Theorem. Precisely, for any {\it convex}
$C^1$--function $\Phi \::\:\mathbb{R}^+ \to \mathbb{R}$,  the
associated so-called $H$--functional (relatively to the equilibrium
$\mathcal{M}$)
\begin{equation} \label{Hphi} H_{\Phi}(f|\mathcal{M})=\int_{\R^3}\mathcal{M}(\v)\,
\Phi\left (\dfrac{f(\v)}{\mathcal{M}(\v)} \right)\d\v\,,
\end{equation}
is decreasing along the flow of the equation \eqref{evol} (this is the opposite
of a physical entropy), with its
associated dissipation functional vanishing only when $f$ is co-linear to the
equilibrium $\mathcal{M}$:
\begin{theo}[\textit{\textbf{Formal $H$--Theorem}}]\label{theoH}
Let $f(t,v) \geq0$ be a space homogeneous solution
then we have formally
\begin{equation}\label{h}
\dfrac{\d}{\d t}H_{\Phi}(f(t)|\mathcal{M})=\int_{\R^3_\v} \Q(f)
(t,\v) \, \Phi'\left (\frac{f(t, \v)}{M(\v)} \right)\d\v \leq 0 \qquad
\qquad (t \geq 0).
\end{equation}
\end{theo}
The application of the $H$-Theorem with $\Phi(x)= (x-1)^2$ suggests
the following Hilbert space setting: the unknown distribution $f$
has to belong to the weighted Hilbert space
$L^2(\mathcal{M}^{-1})=L^2(\R^3_\v\,;\,\mathcal{M}^{-1}(\v) \d\v)$.
Consequently, one defines the Maxwell molecules and hard spheres
collision operators, \textit{associated to the mean-free path
$\lambda=1$}, with their suitable domains, as follows:
\begin{equation*}\begin{cases}
\mathscr{D}(\mathcal{L}_{\mathsf{hs}}) \subset
L^2(\mathcal{M}^{-1}),\,\qquad
\mathrm{Range}(\mathcal{L}_{\mathsf{hs}}) \subset L^2(\mathcal{M}^{-1}),\\
\mathcal{L}_{\mathsf{hs}}f=\Q_{\mathsf{hs}} f\,\qquad \text{ for any
} f \in
\mathscr{D}(\mathcal{L}_{\mathsf{hs}}),\end{cases}\end{equation*}
and
\begin{equation*}\begin{cases}
\mathscr{D}(\mathcal{L}_{\mathsf{max}})=L^2(\mathcal{M}^{-1}),\,\qquad
\mathrm{Range}(\mathcal{L}_{\mathsf{max}}) \subset L^2(\mathcal{M}^{-1}),\\
\mathcal{L}_{\mathsf{max}}f=\Q_{\mathsf{max}} f\,\qquad \text{ for
any } f \in
\mathscr{D}(\mathcal{L}_{\mathsf{max}}).\end{cases}\end{equation*}
Precisely,
$$\mathscr{D}(\mathcal{L}_\mathsf{hs})=\left\{f \in
L^2(\mathcal{M}^{-1})\,;\,\sigma_\mathsf{hs} f \in
L^2(\mathcal{M}^{-1})\right\}$$ where $\sigma_{\mathsf{hs}}$ is the
collision frequency associated to the hard-spheres collision kernel:
$$\sigma_{\mathsf{hs}}(\v)=\dfrac{1}{2\pi}\It |\q \cdot \n|\mathcal{M}_1(\w)\d \w\d \n,\qquad \v \in \R^3.$$
Note that $\sigma_{\mathsf{hs}}$ is unbounded \cite{ArLo}: there
exist
 positive constants $\nu_0,\nu_1$ such that
$$\nu_0(1+|\v-\u|) \leq \sigma_{\mathsf{hs}}(\v) \leq \nu_1 (1+|\v-\u|), \qquad \forall
\v \in \R^3.$$ For this reason,
$\mathscr{D}(\mathcal{L}_\mathsf{hs}) \neq L^2(\mathcal{M}^{-1})$.
On the contrary, the collision frequency $\sigma_{\mathsf{max}}$
associated to the Maxwell molecules collision kernel,
$$\sigma_{\mathsf{max}}(\v)=\dfrac{1}{2\pi}\It |\widetilde{\q} \cdot \n|\mathcal{M}_1(\w)\d \w\d
\n=1$$ is independent of the velocity $v$ and
$\mathcal{L}_\mathsf{max}$ is a bounded operator in
$L^2(\mathcal{M}^{-1})$. 
 We recall (see \cite{ArLo}) that $\mathcal{L}_{\mathsf{hs}}$ is a
negative self--adjoint operator of $L^2(\mathcal{M}^{-1})$.
Moreover, let us introduce the dissipation entropy functionals
associated to $\mathcal{L}_\mathsf{max}$ and
$\mathcal{L}_\mathsf{hs}$:
$$\mathcal{D}_{\mathsf{max}}(f):=-\int_{\R^3} \mathcal{L}_{\mathsf{max}} (f)(\v) \, f(\v) \, \mathcal{M}^{-1}(\v) \,
  \d\v, \qquad f \in L^2(\mathcal{M}^{-1})$$
  and
$$\mathcal{D}_{\mathsf{hs}}(f):= -\int_{\R^3}
\mathcal{L}_{{\mathsf{hs}}} (f)(\v) \, f(\v) \, \mathcal{M}^{-1}(\v)
\,
 \d\v, \qquad f \in \mathscr{D}(\mathcal{L}_\mathsf{hs}).$$
Note that, by virtue of \eqref{h}, if $f(t)$ denotes the (unique)
solution to \eqref{evol} in $L^2(\mathcal{M}^{-1})$ for
hard-spheres interactions, then, with the choice
$\Phi(x)={(x-1)^2}$,
\begin{equation}\label{dissip}
\dfrac{\d}{\d t} \|f(t)-\mathcal{M}\|_{L^2(\mathcal{M}^{-1})}^2
=\dfrac{\d}{\d
t}H_{\Phi}(f(t)|\mathcal{M})=-2\,\mathcal{D}_\mathsf{hs}(f(t)).
\end{equation}  The same occurs for
$\mathcal{D}_\mathsf{max}(f(t))$. This is the reason why we are
looking for a control estimate for both $\mathcal{D}_\mathsf{hs}(f)$
and $\mathcal{D}_\mathsf{max}(f)$ with respect to the
$L^2(\mathcal{M}^{-1})$ norm of $f$. It will be useful to derive an
alternative expression for both $ \mathcal{D}_\mathsf{max}$ and $
\mathcal{D}_\mathsf{hs}$:
\begin{propo}\label{DH}
For any $f \in \mathscr{D}(\mathcal{L}_{{\mathsf{hs}}})$,
 \begin{equation*}
 \mathcal{D}_{\mathsf{hs}}(f) =\frac{1}{4\pi} \, \int_{\R_\v^3 \times \R_\w^3 \times \S}
       {\qn}  \left[\frac{f(\v^\star)}{\mathcal{M}(\v^\star)} -\frac{f(\v )}{\mathcal{M}(\v )}  \right]^2
        \mathcal{M}_1(\w) \mathcal{M} (\v)\, \d\w \, \d\v \, \d\n \geq 0.
  \end{equation*}
In the same way,
\begin{equation*}
  \mathcal{D}_{\mathsf{max}}(f)=\frac{1}{4\pi} \, \int_{\R_\v^3 \times \R_\w^3 \times \S}
       |\widetilde{\q} \cdot \n|  \left[\frac{f(\v^\star)}{\mathcal{M}(\v^\star)} -\frac{f(\v )}{\mathcal{M}(\v )}  \right]^2
        \mathcal{M}_1(\w) \mathcal{M} (\v)\, \d\w \, \d\v \, \d\n \geq 0.
  \end{equation*}
for any $f \in L^2(\mathcal{M}^{-1}).$ \end{propo}

\begin{proof} The proof is a straightforward particular case of the above ${H}$-Theorem.
Precisely, let us fix $f \in
\mathscr{D}(\mathcal{L}_{\mathsf{hs}})$, and set $f=g \mathcal{M}$ then one has
$$\int_{\R^3_\v} \mathcal{L}_{\mathsf{hs}} (f) \, f \,
\mathcal{M}^{-1}  \, \d\v =
     \int_{\R^3 \times \R^3 \times \S}
        \dfrac{\qn}{2\pi\sigma} \mathcal{M}(\v) \mathcal{M}_1(\w) g(\v) \left[ g(\v^\star) - g(\v) \right] \, \d\w \, \d\v \,
        \d\n.$$
Moreover, for any $\varphi \in
\mathscr{D}(\mathcal{L}_{\mathsf{hs}})$, since $\mathcal{M}$ is an
equilibrium state of $\Q_{\mathsf{hs}}$,
$$\int_{\R^3 \times \R^3 \times \S}
       \dfrac{\qn}{2\pi\sigma}\mathcal{M}(\v) \mathcal{M}_1(\w) \left[ \varphi(\v^\star) - \varphi(\v) \right] \, \d\w \, \d\v \, \d\n
     = \langle \mathcal{L}_{\mathsf{hs}}(\mathcal{M}), \varphi \rangle_{L^2(\R,\d\v)}
     =0.$$
It is easy to deduce then that
  $$\int_{\R_\v^3} \mathcal{L}_{\mathsf{hs}} (f)  \, f  \,
\mathcal{M}^{-1} \, \d\v = \int_{\R^3 \times \R^3 \times \S}
  \dfrac{\qn}{2\pi\sigma}\mathcal{M}(\v) \mathcal{M}_1(\w) g(\v^\star) \left[ g(\v) - g(\v^\star) \right] \,\d\w \, \d\v \,
  \d\n.$$
Finally taking the mean of these two quantities leads to the desired
result. The same reasoning also holds for
$\mathcal{D}_{\mathsf{max}}$.\end{proof}

\section{Quantitative estimates of the spectral gap}\label{sec:entropy} \setcounter{equation}{0}

In this section, we strengthen the above result in providing a
quantitative lower bound for both $\mathcal{D}_{\mathsf{max}}(f)$
and $\mathcal{D}_{\mathsf{hs}}(f)$. The estimate for
$\mathcal{D}_{\mathsf{max}}(f)$ is related to the spectral
properties of the collision operator $\mathcal{L}_\mathsf{max}$
while that of $\mathcal{D}_\mathsf{hs}(f)$ relies on a suitable
comparison with $\mathcal{D}_\mathsf{max}(f).$ From now on, $\langle
\cdot,\cdot \rangle$ denotes the scalar product of
$L^2(\mathcal{M}^{-1})$ and $\mathfrak{S}(\mathcal{L})$,
$\mathfrak{S}_p(\mathcal{L})$ shall denote respectively the spectrum
and the point spectrum of a given (non necessarily bounded) operator
in $L^2(\mathcal{M}^{-1})$.

\subsection{Spectral study for Maxwell molecules}\label{Sec:spectre}

We already saw that the operator $\mathcal{L}_{\mathsf{max}}: L^2(\mathcal{M}^{-1})
\to L^2(\mathcal{M}^{-1})$ is bounded, and it is easily seen that
$\mathcal{L}_\mathsf{max}$ splits as
$$\mathcal{L}_\mathsf{max}f =\mathcal{L}^+_\mathsf{max}f -
 f(\v)$$ where $\mathcal{L}^+_\mathsf{max}$ is
compact and self-adjoint (the proof can be done similarly as in \cite{ArLo})
and we used that the collision frequency
$\sigma_{\mathsf{max}}$ associated to the Maxwell molecules is
constant.  Moreover, since
$$\langle \mathcal{L}_\mathsf{max}f,f\rangle \leq 0,\qquad \forall f
\in L^2(\mathcal{M}^{-1}),$$ the operator is easily seen to generate
a $C^0$-semigroup of contractions, and it is known that
$\mathfrak{S}(\mathcal{L}_\mathsf{max}) \subset (-\infty,0].$
Finally, since
$\mbox{Id}+\mathcal{L}_\mathsf{max}=\mathcal{L}_\mathsf{max}^+$ is a
positive, self-adjoint compact operator, one sees that the spectrum
of $\mathcal{L}_\mathsf{max}$ is made of a discrete set of
eigenvalues with finite algebraic multiplicities plus possibly
$\{-1\}$ in the essential spectrum, with
\begin{equation}\label{spec-10}
\mathfrak{S}_p(\mathcal{L}_\mathsf{max})
\subset (-1,0]\end{equation} and where the only possible accumulation
point is $\{-1\}$. Clearly, since
$\mathcal{L}_\mathsf{max}(\mathcal{M})=0,$ $\lambda_{0,0}:=0$ is an
eigenvalue of $\mathcal{L}_\mathsf{max}$ with eigenspace given by
$\mathrm{Span}(\mathcal{M})$. There are other eigenvalues of
$\mathcal{L}_\mathsf{max}$ of peculiar interest. Namely, for any $f
\in L^2(\mathcal{M}^{-1})$, the weak formulation \eqref{weak} yields
\begin{equation*}\begin{split}
\int_{\R^3_\v}  (\v-u_1)\mathcal{L}_{\mathsf{max}} (f ) \,
\d\v&=\dfrac{1}{2\pi}\int_{\R_v^3 \times \R_w^3 \times \mathbb{S}^2}
|\widetilde{q} \cdot \n| f(   \v)
\mathcal{M}_1(\w)\left[v^\star-\v\right]\d v \d w \d n\\
&=-\dfrac{\alpha(1-\beta)}{\pi}\int_{\R_v^3}f(v)\d v
\int_{\R^3_w}\mathcal{M}_1(\w)\d\w\int_{\mathbb{S}^2}|\widetilde{q}\cdot
\n| (q \cdot \n)\,\n \d n.
\end{split}\end{equation*}
Using the fact that $\ds \int_{\mathbb{S}^2}|\widetilde{q}\cdot \n|
(q \cdot \n)\,\n \d n=\pi \, q,$ one has
\begin{equation}\label{integralv}\begin{split}\int_{\R_\v^3} (\v-u_1) \mathcal{L}_{\mathsf{max}} (f) \,
\d\v&=-\alpha
(1-\beta)\int_{\R_v^3}f (v)\d\v\int_{\R_w^3}(\v-\w)\mathcal{M}_1(\w)\d\w\\
&=-\alpha (1-\beta)\int_{\R_v^3} (\v-u_1) f (v)\d\v.
\end{split}\end{equation} The operator
$\mathcal{L}_\mathsf{max}$ being self-adjoint in
$L^2(\mathcal{M}^{-1} )$, one obtains the identity
$$\left\langle f,\mathcal{L}_\mathsf{max}((v_i-u_{1,i})\mathcal{M})\right\rangle =-\a(1-\b)
\left\langle f, (v_i-u_{1,i}) \mathcal{M} \right \rangle,  \qquad
i=1,2,3,\quad \forall f \in L^2(\mathcal{M}^{-1} ).$$ If we denote
$$\lambda_{0,1}= \alpha (1-\beta)=\kappa \in (0,1),$$
this means that $-\lambda_{0,1}$ is an {eigenvalue of}
$\mathcal{L}_\mathsf{max}$ associated to the
\textit{\textbf{momentum eigenvectors}} $(v_i-u_{1,i})
\mathcal{M}(v)$ for any $i=1,2,3.$ In the same way,  technical
calculations show that, for any $f \in L^2(\mathcal{M}^{-1})$,
\begin{equation*}\begin{split}
\langle \mathcal{L}_\mathsf{max}(f)\,,&\,|v-u_1|^2
\mathcal{M}\rangle=\int_{\R^3}\mathcal{L}_\mathsf{max}(f)|v-u_1|^2 \d\v\\
&=-2\a(1-\b)(1-\a(1-\b))\langle f, |v-u_1|^2\mathcal{M}\rangle  +
\dfrac{6\Theta_1}{m_1}\a^2(1-\b)^2 \langle
f,\mathcal{M}\rangle,\end{split}\end{equation*} and, since
$\mathcal{L}_\mathsf{max}$ is self-adjoint, setting
$$\lambda_{1,0}=2\a(1-\b)(1-\a(1-\b))=2\kappa(1-\kappa),$$
one has $\langle \mathcal{L}_\mathsf{max}(|v-u_1|^2\mathcal{M}),
f\rangle = -\lambda_{1,0} \langle |v-u_1|^2\mathcal{M}, f \rangle,$
for any $f \,\bot \,\mathrm{span}(\mathcal{M}).$  Equivalently,
$$\mathcal{L}_\mathsf{max}\left(\mathcal{E}\right)=-\lambda_{1,0} \mathcal{E},$$
where
$\mathcal{E}(v)=\left(|v-u_1|^2-\frac{3\Theta^{\#}}{m}\right)\mathcal{M},$
is the \textbf{\textit{ energy eigenfunction}}, associated to
$-\lambda_{1,0}.$ To summarize, we obtained three particular
eigenvalues $\lambda_{0,0},\lambda_{0,1}$ and $\lambda_{1,0}$ of
$\mathcal{L}_\mathsf{max}$ associated respectively to the
equilibrium, momentum and energy eigenfunctions.

To provide a full picture of the spectrum of
$\mathcal{L}_{\mathsf{max}}$, we adopt the strategy of Bobylev
\cite{Bo88} based on the application of the Fourier transform to the
elastic Boltzmann equation. Namely, we are looking for $\lambda >0$
 such that the equation
\begin{equation}\label{spect}
\mathcal{L}_\mathrm{max} f =-\lambda f,\qquad f \in
L^2(\mathcal{M}^{-1}), \:f \neq 0,\end{equation} admits a solution.
Applying the Fourier transform $\mathscr{F}$ to the both sides of
the above equation, we are lead to:
$$\mathscr{F}\left[\mathcal{L}_\mathrm{max} f\right] (\xi)=-\lambda \widehat{f}(\xi),\qquad \xi \in \R^3,$$
where $\widehat{f}=\mathscr{F}(f)$. One deduces immediately from the
calculations performed in \cite{SpTo}, that
  \[ \mathscr{F}\left[\mathcal{L}_\mathrm{max} f\right] (\xi) = \frac{1}{2 \pi} \,
       \int _{\S} |\widetilde{\xi} \cdot \n|
         \left[ \widehat{f}(\xi^+) \widehat{\mathcal{M}_1} (\xi^-) - \widehat{f}(\xi) \widehat{\mathcal{M}_1} (0) \right] \, \d\n \]
with $\widetilde{\xi}=\xi/|\xi|$ and
$$\xi^+ = \xi - 2\alpha(1-\beta) (\xi \cdot \n)\n, \qquad  \qquad
  \xi^- = 2\alpha(1-\beta) (\xi \cdot \n)\n$$
while $\widehat{\mathcal{M}_1}$ is the Fourier transform of the
background Maxwellian distribution, given by
  \[ \widehat{\mathcal{M}_1}(\xi) = \exp \left\{ -i u_1 \cdot \xi - \frac{ \Theta_1 |\xi|^2}{2 m_1} \right\}. \]
The fundamental property  is that even if the equilibrium
distribution $\mathcal{M}$ does not make the integrand of the
collision operator vanish pointwise (as it is the case in the
elastic case),  surprisingly it still satisfies a pointwise relation
in Fourier variables as was noticed in~\cite{SpTo}. A simple
computation yields
  \[ \widehat{\mathcal{M}}(\xi) = \exp \left\{ -i u_1 \cdot \xi - \frac{ \Theta^{\#} |\xi|^2}{2 m} \right\} \]
and thus one checks easily that,  $\widehat{\mathcal{M}}(\xi^+) \,
\widehat{\mathcal{M}_1} (\xi^-) = \widehat{\mathcal{M}} (\xi),$ for
any $n \in \mathbb{S}^2,$ and any $\xi \in \R^3.$ Thus, following
the method of Bobylev~\cite{Bo88}, we rescale $\widehat{f}(\xi) =
\widehat{\mathcal{M}}(\xi) \, \varphi(\xi)$ and define the
corresponding re-scaled operator $\mathscr{L} $
  \[\mathscr{L} \varphi(\xi) = \frac{1}{2 \pi} \,
       \int _{\S} |\widetilde{\xi} \cdot \n|
         \left[ \varphi(\xi^+) - \varphi(\xi) \right] \, \d\n. \]
Then, Eq. \eqref{spect} amounts to find  $\lambda >0$ such that the
equation $\mathscr{L}  \varphi(\xi)=-\lambda \varphi(\xi)$  admits a
non zero solution $\varphi$ with
\begin{equation}\label{admi}
f=\mathscr{F}^{-1}(\varphi \widehat{\mathcal{M}}^{-1}) \in
L^2(\mathcal{M}^{-1})\end{equation} where $\mathscr{F}^{-1}$ stands
for the inverse Fourier transform. Using, as in the elastic case
\cite[p. 136]{Bo88}, the symmetry properties of the operator
$\mathscr{L}$ together with condition \eqref{admi}, one obtains that
functions of the form
  \[ \varphi_{n,\l, m} (\xi) = |\xi|^{2n+\l} \, \mathbf{Y}_{\l, m} (\widetilde{ \xi } ), \qquad n \geq 0, \]
 $\mathbf{Y}_{\l, m}$  being a spherical harmonic ($\l \in
\mathbb{N}$, $m=-\l,\ldots,\l$),  are eigenvectors of $\mathscr{L}$,  
associated to the
eigenvalues $-\lambda_{n,\l}=-\lambda_{n,\l}(\kappa)$ where,
according to \cite{BCG},
\begin{equation}\label{eigenvaluep}\lambda_{n,\ell}=1 -
\frac{1}{2\kappa(1-\kappa)}
   \int_{1-2\kappa} ^1 s^{2n+\l+1}\,
     \mathbf{P}_\l \left(\frac{1-2\kappa+s^2}{(2-2\kappa)s}\right) \, \d
     s\end{equation}
where $\mathbf{P}_\l$ is the $\l$-th Legendre polynomial, $n \in
\mathbb{N},$ $\l \in \mathbb{N}.$ Note that, according to
\eqref{spec-10}, $\lambda_{n,\l} \in (0,1].$
 Technical calculations prove that the eigenvalues we already found
out, namely,
$$\lambda_{0,0}=0, \quad \lambda_{0,1}=\kappa, \quad \text{ and }
\quad \lambda_{1,0}=2\kappa(1-\kappa),$$ do actually correspond to
the couples $(n,\l)=(0,0)\,;(0,1)$ and $ (1,0)$
 respectively.
  Moreover, from
the well-known Legendre polynomials
 property:
$$(\l+1)\mathbf{P}_{\l+1}(x)=(2\l+1) x
\mathbf{P}_{\l}(x)-\l\,\mathbf{P}_{\l-1}(x),\qquad x \in \mathbb{R},
\,\l \geq 1,$$ one obtains the recurrence formula
\begin{equation}\label{recurrence}\lambda_{n,\l+1}=\dfrac{2\l+1}{\l+1}\dfrac{\nu}{1+\nu}
\lambda_{n,\l} + \dfrac{2\l+1}{
(\l+1)(1+\nu)}\lambda_{n+1,\l}-\dfrac{\l}{\l+1}\lambda_{n+1,\l-1},\qquad
n, \l \geq 0\end{equation} where $\nu=1-2\kappa \in (-1,1)$. Such a
recurrence formula together with the relation
\begin{equation*}
\lambda_{n,0}=1-\dfrac{1}{n+1}\dfrac{1-\nu^{2n+2}}{1-\nu^2},\end{equation*}
 allow to prove by induction over $\l
\in \mathbb{N}$ that
$$\lambda_{n+1,\l} \geq \lambda_{n,\l}\quad \text{ and }\quad \lambda_{n,\l+1} \geq \lambda_{n,\l}
\quad \text{ for any } \quad n \in \mathbb{N}.$$
Consequently, one sees that $\min\{\lambda_{n,\l}\,;\, n, \l \geq
0\} \setminus \{0\} = \min\{\lambda_{1,0};\lambda_{0,1}\}$ which means
that  spectral gap of $\mathcal{L}_\mathsf{max}$ is given by
$$\mu_{\mathsf{max}}=\min\{\lambda_{1,0};\lambda_{0,1}\}=\min\{\kappa;2\kappa(1-\kappa)\}.$$

 \begin{nb} Note that,
 $$\lambda_{0,1} \leq \lambda_{1,0} \Longleftrightarrow \kappa \leq 1/2 \Longleftrightarrow \a(1+\e)  \leq 1.$$
In particular, if $m_1 \leq m$ then  $\lambda_{0,1} \leq
\lambda_{1,0}$. Assuming for a while that we are dealing with
species of gases with same masses $m=m_1$, then  in the true
inelastic case ({\em i.e.}, $\e <1$) one also has $\lambda_{1,0} >
\lambda_{0,1}$. This situation is very particular to inelastic
scattering and means that the cooling process of the temperature
happens more rapidly than the forcing of the momentum by the
background, whereas when $\kappa=1/2$ these two processes happen at
exactly the same speed $(\lambda_{0,1}=\lambda_{1,0})$. Note also
that, whenever $\lambda_{0,1} > \lambda_{1,0}$ (due to the ratio
of mass different from $1$), the smallest
eigenvalue corresponds to the momentum relaxation
and not the energy relaxation anymore.
This contrasts very much with the linearized case (see~\cite{BCG}).
Note also that the first eigenvalues are ordered as illustrated in
Fig.~1.
\end{nb}

The above result can be summarized in the following where the last
statement follows from the fact that $\mathcal{L}_\mathsf{max}+ \mbox{Id}$
is a self-adjoint compact operator of $L^2(\mathcal{M}^{-1})$.
\begin{theo}\label{Spec}The operator $-\mathcal{L}_\mathsf{max}$ is a bounded self-adjoint
positive operator of $L^2(\mathcal{M}^{-1})$ whose spectrum is composed of
an essential part $\{+1\}$ plus the following discrete part:
$$\mathfrak{S}_p(-\mathcal{L}_\mathsf{max})=\left\{\lambda_{n,\l}\,;\,n,\,\l
\in \mathbb{N}\right\} \subset [0,1)$$ where $\lambda_{n,\l}$ is
given by \eqref{eigenvaluep}. Moreover, $\lambda_{0,0}=0$ is a
simple eigenvalue of $\mathcal{L}_\mathsf{max}$ associated to the
eigenvector $\mathcal{M}$ and $-\mathcal{L}_\mathsf{max}$ admits a
positive spectral gap
$$\mu_{\mathsf{max}}=\min\{\lambda_{1,0};\lambda_{0,1}\}=\min\{\kappa;2\kappa(1-\kappa)\}.$$
Finally, there exists a Hilbert basis of $L^2(\mathcal{M}^{-1})$
made of eigenvectors of $-\mathcal{L}_\mathsf{max}.$
\end{theo}

\begin{figure}[h] 
  \begin{center}
  \includegraphics[scale=.9]{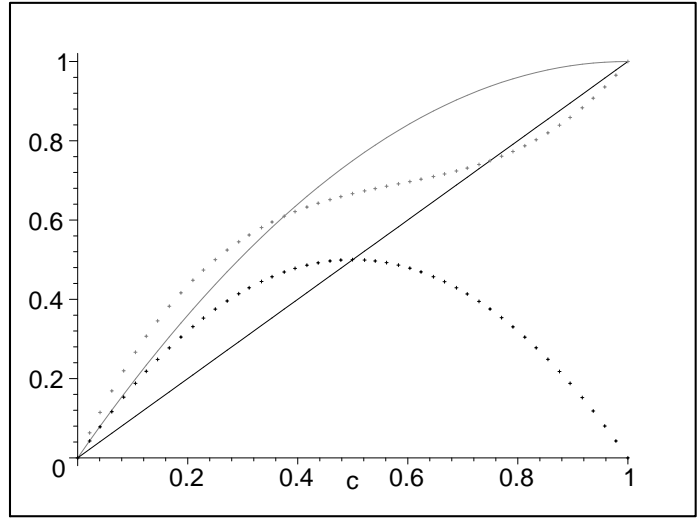}
  \caption{Evolution of $\lambda_{0,1}$ (black solid line), $\lambda_{1,0}$ (black dotted line), $\lambda_{1,1}$ (grey dotted line) and
  $\lambda_{0,2}$ (grey solid line) as functions
  of $c \in [0,1]$}\label{carl}
  \end{center}
  \end{figure}

\subsection{Entropy estimate for Maxwell molecules}

The result of the above section allows to provide a quantitative
version of the $H$-Theorem. Precisely, for any $f \in
L^2(\mathcal{M}^{-1})$ orthogonal to $\mathcal{M}$, using the
decomposition of both $\mathcal{L}_\mathsf{max}f$ and $f$ on the
Hilbert basis of $L^2(\mathcal{M}^{-1})$ made of eigenvectors of
$-\mathcal{L}_\mathsf{max}$ (see Theorem \ref{Spec}), it is easily
proved that:
  \begin{equation}\label{nullspace}
     \mathcal{D}_{\mathsf{max}}(f):= - \int_{\R_\v^3} \mathcal{L}_{\mathsf{max}} (f) \, f \, \mathcal{M}^{-1} \, \d\v
     \geq  \mu_\mathsf{max} \, \|f\|_{L^2(\mathcal{M}^{-1})} ^2,
     \qquad \forall f \bot \, \mathrm{Span}(\mathcal{M}).
     \end{equation}
It is well-known that such a {\it coercivity estimate} allows to
obtain an exponential relaxation rate to equilibrium for the
solution to the space homogeneous Boltzmann equation. Namely, given
$f_0(\v) \in L^2(\R^3,\mathcal{M}^{-1}(\v)\d\v)$ with unit mass
$$\int_{\R^3} f_0(v)\d v=1,$$
let $f_t$ be the unique solution of \eqref{evol} with initial
condition $f_{t=0}=f_0.$ According to the conservation of mass, it
is clear that $(f_t -\mathcal{M})$ is  orthogonal to $\mathcal{M}$
(for the $L^2(\R_v^3,\mathcal{M}^{-1}(\v)\d\v)$ scalar product) and,
within the entropy language:
  \[ \frac{\d}{\d t} H_\Phi (f_t | \mathcal{M})=-2 \mathcal{D}_\mathsf{max}(f_t) \le - 2 \mu_\mathsf{max} \, H_\Phi (f_t | \mathcal{M}) \]
for $\Phi(x)=(x-1)^2$ or equivalently,
  \[ \left( \int_{\R^3} (f_t-\mathcal{M})^2 \, \mathcal{M}^{-1} \, \d\v \right)^{1/2} \le
      \left( \int_{\R^3} (f_0-\mathcal{M})^2 \, \mathcal{M}^{-1} \, \d\v \right)^{1/2} \, \exp\left(-\mu_\mathsf{max} t\right), \qquad \forall t \geq 0. \]

We obtain in this way an {\it explicit} exponential relaxation rate
towards equilibrium for the solution to the space homogeneous linear
Boltzmann equation  which is valid for granular gases of Maxwell
molecules and generalizes a well-known result for classical gases
\cite{Ce88}. More interesting is the fact that the knowledge of the
spectral gap of $\mathcal{L}_\mathsf{max}$ allows to recover an
explicit estimate of the spectral gap of the linear Boltzmann
operator for hard-spheres $\mathcal{L}_\mathsf{hs}$ through a
suitable comparison of the entropy production functionals
$\mathcal{D}_{{\mathsf{hs}}}$ and $\mathcal{D}_{{\mathsf{max}}}$.
This is the subject of the following section.

\subsection{Entropy estimate for hard-spheres}

The goal of this subsection is to show that the entropy production
functional $\mathcal{D}_{{\mathsf{hs}}}$ for hard-spheres relates to
the one for Maxwell molecules $\mathcal{D}_\mathsf{max}$. More
precisely we shall show that
  \begin{propo} \label{linkHM}
 The entropy production functionals $\mathcal{D}_\mathsf{hs}$ and
 $\mathcal{D}_\mathsf{max}$ are related by:
    \[ \mathcal{D}_{{\mathsf{hs}}}(f) \geq C^\star \, \mathcal{D}_{\mathsf{max}}(f), \qquad \forall f \in \mathscr{D}(\mathcal{L}_\mathsf{hs}) \]
  for some explicit constant $C^\star$ depending only on $\alpha$ and $\beta$.
  \end{propo}

\begin{nb} The idea of searching for such an inequality was already
present in~\cite{BaMo}, but here the method of proof is different
and simpler: one does not need any triangular inequality between
collisions, and the proof reduces to a careful study of a
convolution integral.\end{nb}

\begin{nb} Note that in the hard-spheres case, the operator $\mathcal{L}_\mathsf{hs}$ is unbounded.
For a careful study of its properties (compactness of the non-local part,
definition of the associated $C^0$-semigroup of contraction in the
Hilbert space $L^2(\mathcal{M}^{-1})$) we refer to \cite{ArLo}.
\end{nb}

\begin{proof} Let $f \in
\mathscr{D}(\mathcal{L}_{\mathsf{hs}})$. We set $u_1=0$ in this proof
without restriction since this only amounts to a space translation.

We introduce the following parametrization, for fixed $\n \in \S$, $\v=r\n+\bar{v}$,
$\v^\star=r^\star\n+\bar{v}$, $\w=r_\w \n + \bar{w}$,
$\w^\star=r_{\w^\star} \n + \bar{w}$, where $r, r^\star,r_\w$ and
$r_{\w^\star}$ are real numbers and $\bar{v},\bar{w} $ are
orthogonal to $\n$. Simple computations show that
$$r_\w = \frac{r^\star}{2\alpha(1-\beta)} +
\left(1-\frac{1}{2\alpha(1-\beta)}\right) r,
$$
while
$$r_{\w^\star} =
\left(\frac{1}{2\alpha(1-\beta)}-\frac{1-\alpha}{\alpha}\right)r^\star
+ \left(\frac{1}{\alpha}-\frac{1}{2\alpha(1-\beta)} \right) r.
$$
Therefore, $r_\w$ and $r_{\w^\star}$ only depend on $r$ and
$r^\star$. Then if we denote $\theta$ the angle between
$\widetilde{\q}$ and $\n$, we get from Prop. \ref{DH}, where we set
$g=\frac{f}{\mathcal{M}}$,
  \begin{multline*}
  \mathcal{D}_{{\mathsf{hs}}}(f) = \frac{1}{2\pi  } \, \int_{\S} \int_{r,r^\star \in \R} \int_{\bar{v},\bar{w} \in \n^\bot}
                  |\q| \, \cos \theta \,
                  \Big[ g(r^\star\n +\bar{v}) - g(r \n +\bar{v}) \Big]^2 \\ \mathcal{M}_1(r_\w \n +\bar{w})
     \, \mathcal{M} (r\n+\bar{v}) \d \bar{v} \, \d \bar{w} \, \d r \, \d r^\star \, \d\n
  \end{multline*}
with
$$|\q| = \Big(|\bar{v}-\bar{w}|^2+
(2\kappa)^{-2}|r^\star-r|^2\Big)^{1/2}
$$
and
$$\cos \theta = \frac{(2\kappa)^{-1}|r^\star-r|}
      {\left(|\bar{v}-\bar{w}|^2+(2\kappa)^{-2}|r^\star-r|^2\right)^{1/2}},
$$
      where we recall that $\kappa=\a(1-\b)$. We split the integral into two parts according to
      $|r^\star-r|\ge
\varrho_0
>0$ or $|r^\star-r| \le \varrho_0$ where $\varrho_0$ is a positive parameter to be determine latter. Using the fact that $|\q| \ge
|r^\star-r|/2\kappa$, one has the following estimate for the first
part of the integral
\begin{equation*}\begin{split}
\frac{1}{2\pi } \,& \int_{\S} \d\n \int_{\{|r^\star-r|\ge
\varrho_0\}}\d r \, \d
r^\star \int_{\bar{v},\bar{w} \in \n^\bot}
                 |\q| \, \cos \theta \, \mathcal{M} (r\n+\bar{v}) \, \mathcal{M}_1(r_\w \n +\bar{w})\\
                               &\phantom{+++++ ++++++++}
                \Big[ g(r^\star\n +\bar{v}) -g(r \n +\bar{v}) \Big]^2 \,\d \bar{v} \,  \d \bar{w} \\
&\geq\frac{(2\kappa)^{-1} \, \varrho_0}{2\pi  } \,
       \int_{\S}\d\n \int_{\{|r^\star-r|\geq \varrho_0\}} \, \d r \, \d r^\star\\
       &\int_{\bar{v},\bar{w} \in \n^\bot}
               \cos \theta \, \mathcal{M} (r\n+\bar{v}) \, \mathcal{M}_1(r_\w \n +\bar{w})  \Big[g(r^\star\n +\bar{v}) - g(r \n +\bar{v}) \Big]^2 \, \d \bar{v} \, \d \bar{w}
\end{split}  \end{equation*} which corresponds (up to the multiplicative factor $(2\kappa)^{-1} \, \varrho_0$)
to the integral for $|r-r^\star| \geq \varrho_0$ corresponding to
Maxwell molecules, {\em i.e.},
\begin{equation}\label{geq0}
-\int_{\R^3} \chi_{\{|r-r^\star| \geq \varrho_0\}}
\mathcal{L}_\mathsf{hs}(f) f \mathcal{M}^{-1}\d \v \geq
-\frac{\varrho_0}{2\kappa}\int_{\R^3} \chi_{\{|r-r^\star| \geq
\varrho_0\}} \mathcal{L}_\mathsf{max}(f) f \mathcal{M}^{-1}\d
\v.\end{equation}

Concerning now the second part of the integral (corresponding to
$|r^\star-r| \le \varrho_0$), we use that $|\q| \ge
|\bar{v}-\bar{w}|$ and we isolate the integration over $\bar{w}$:
  \begin{equation*}\begin{split}
 -\int_{\R^3}\,& \chi_{\{|r-r^\star| \leq\varrho_0\}}
\mathcal{L}_\mathsf{hs}(f) f \mathcal{M}^{-1}\d \v\geq \frac{1}{2\pi
} \, \int_{\S}\d\n \int_{\{|r^\star-r|\le \varrho_0\}}\d r\,\d r^\star \, |r^\star-r|
\\&
\int_{\bar{v} \in \n^\bot}
 (2\kappa)^{-1} \left( \frac{m_1}{2 \pi \Theta_1} \right)^{-3/2} \left( \int_{\bar{w} \in \n^\bot}
    \frac{|\bar{v}-\bar{w}| \, \mathcal{M}_1(\bar{w})}
    {\left(|\bar{v}-\bar{w}|^2+ (2\kappa)^{-2}|r^\star-r|^2\right)^{1/2}} \, \d \bar{w} \right) \\
&\phantom{+++++}   \mathcal{M} (r\n+\bar{v}) \, \mathcal{M}_1(r_\w
\n) \, \Big[ g(r^\star\n +\bar{v}) - g(r \n +\bar{v}) \Big]^2 \, \d
\bar{v}
 \end{split}  \end{equation*}
where we used the fact that, since $\bar{w}$ is orthogonal to $\n$,
$$
\mathcal{M}_1(r_\w \n+\bar{w})= \left( \frac{m_1}{2 \pi \Theta_1}
\right)^{-3/2} \mathcal{M}_1(\bar{w})\mathcal{M}_1(r_\w \n).
$$
Setting $\xi=|r^\star-r|/2\kappa $, if one were able to prove that there is
a constant $C$ such that \begin{equation}\label{convol} \int_{\R^2}
\frac{|\bar{v}-\bar{w}| \,
\mathcal{M}_1(\bar{w})}{\left(|\bar{v}-\bar{w}|^2+
\xi^2\right)^{1/2}} \, \d \bar{w}
     \geq C \, \int_{\R^2} \frac{\mathcal{M}_1(\bar{w})}{\left(|\bar{v}-\bar{w}|^2+ \xi^2\right)^{1/2}} \, \d
     \bar{w}\end{equation}
uniformly for $\bar{v} \in \R^2$ and $\xi\in [0,
\varrho_0/2\kappa]$, then one would obtain the desired estimate
(by doing all the previous transformations backward):
$$-\int_{\R^3} \chi_{\{|r-r^\star| \leq \varrho_0\}}
\mathcal{L}_\mathsf{hs}(f) f \mathcal{M}^{-1}\d \v \geq -C
\int_{\R^3} \chi_{\{|r-r^\star| \leq \varrho_0\}}
\mathcal{L}_\mathsf{max}(f) f \mathcal{M}^{-1}\d \v.$$

To study the convolution integral of \eqref{convol}, we make a
second splitting between $|\bar{w}-\bar{v}| \geq \varrho_1 >0$ and
$|\bar{w}-\bar{v}| \le \varrho_1$ (for some $\varrho_1
>0$). It gives
  \begin{multline*}
  \int_{\bar{w} \in \R^2} \frac{|\bar{v}-\bar{w}| \, \mathcal{M}_1(\bar{w})}{\left(|\bar{v}-\bar{w}|^2+ \xi^2\right)^{1/2}} \, \d \bar{w}  \\
     \geq \int_{\{|\bar{w}-\bar{v}| \geq \varrho_1\}} \frac{|\bar{v}-\bar{w}| \, \mathcal{M}_1(\bar{w})}{\left(|\bar{v}-\bar{w}|^2+ \xi^2\right)^{1/2}} \, \d \bar{w}
      \geq \varrho_1 \, \int_{\{|\bar{w}-\bar{v}| \geq \varrho_1\}} \frac{\mathcal{M}_1(\bar{w})}{\left(|\bar{v}-\bar{w}|^2+ \xi^2\right)^{1/2}} \, \d \bar{w} \\
      \geq \varrho_1 \, \Bigg( \int_{\R^2} \frac{\mathcal{M}_1(\bar{w})}{\left(|\bar{v}-\bar{w}|^2+ \xi^2\right)^{1/2}} \, \d \bar{w}
      - \int_{\{|\bar{w}-\bar{v}| \le \varrho_1\}} \frac{\mathcal{M}_1(\bar{w})}{\left(|\bar{v}-\bar{w}|^2+ \xi^2\right)^{1/2}} \, \d \bar{w} \Bigg).
  \end{multline*}
Then we use the obvious estimates
$$ \forall \, \bar{v} \in \R^2, \ \forall \, \xi \in [0, \varrho_0/2\kappa],
       \int_{\R^2} \frac{\mathcal{M}_1(\bar{w})}{\left(|\bar{v}-\bar{w}|^2+ \xi^2\right)^{1/2}} \, \d \bar{w} \ge
        \frac{C(\kappa,\varrho_0)}{1+|\bar{v}|}$$
for an explicit constant $C(\kappa,\varrho_0) >0$ depending only on
$\kappa,\varrho_0$, and
  $$\forall \, \bar{v} \in \R^2, \ \forall \,\xi \in [0,\varrho_0/2\kappa],
    \int_{\{|\bar{w}-\bar{v}| \le \varrho_1\}} \frac{\mathcal{M}_1(\bar{w})}{\left(|\bar{v}-\bar{w}|^2+ \xi^2\right)^{1/2}} \, \d \bar{w}
    \le C(\kappa,\varrho_1) \, e^{-|\bar{v}|^2}.$$
for an explicit constant $C(\kappa,\varrho_1) >0$ going to $0$ as
$\varrho_1$ goes to $0$. It yields for $\varrho_1$ small enough
(depending on $\varrho_0$)
\begin{equation}\label{estimate}
\int_{\{|\bar{w}-\bar{v}| \le \varrho_1\}}
\frac{\mathcal{M}_1(\bar{w})}{\left(|\bar{v}-\bar{w}|^2+
\xi^2\right)^{1/2}} \, \d \bar{w}
    \le \frac12 \,  \int_{\R^2} \frac{\mathcal{M}_1(\bar{w})}{\left(|\bar{v}-\bar{w}|^2+ \xi^2\right)^{1/2}} \, \d \bar{w}.
\end{equation}
  for  any $\bar{v} \in \R^2, \xi \in [0,\varrho_0/2\kappa]$ (we
  also refer to the Appendix A of this paper for a construction
  of the parameter $\varrho_1$).
   Consequently, for this choice of $\varrho_1$ we
  obtain \eqref{convol} with $C=\varrho_1/2$, {\em i.e.},
$$-\int_{\R^3} \chi_{\{|r-r^\star| \leq \varrho_0\}}
\mathcal{L}_\mathsf{hs}(f) f \mathcal{M}^{-1}\d \v \geq
-\frac{\varrho_1}{2}\int_{\R^3} \chi_{\{|r-r^\star| \leq
\varrho_0\}} \mathcal{L}_\mathsf{max}(f) f \mathcal{M}^{-1}\d \v.$$
This, together with estimate \eqref{geq0}, yield
  \[ \mathcal{D}_{{\mathsf{hs}}}(f) \geq \min \left\{ \frac{\varrho_0}{2\kappa}, \frac{\varrho_1}{2} \right\} \, \mathcal{D}_{\mathsf{max}}(f) \]
which concludes the proof.
\end{proof}

\begin{nb}\label{nb:lower} The constant $C^\star$ from the proof can be optimized according
to the parameter $\varrho_0$, by expliciting $\varrho_1$ as a
function of $\varrho_0$. Precisely, making use of Lemma A.1 given in
the Appendix,
$$C^\star=\min \left\{ \frac{\varrho_0}{2\kappa}, \frac{\varrho_1}{2}
\right\} \geq \dfrac{\eta}{\sqrt{5}}$$ with $\eta=
\sqrt{\frac{2\Theta_1}{m_1}}
\, \mathrm{erf}^{-1}\left(\frac{1}{2}\right)$ where $\mathrm{erf}^{-1}$
denotes the inverse error function,
$\mathrm{erf}^{-1}(\frac{1}{2})\simeq 0.4769.$ Notice that this
lower bound for $C^\star$ does not depend on the parameters $\a,\b$.\end{nb}

\begin{nb}
The above Proposition provides an estimate of the spectral gap of
$\mathcal{L}_\mathsf{hs}$ in $L^2(\mathcal{M}^{-1})$. Precisely, we
recall from \cite{ArLo} that the spectrum of
$\mathcal{L}_\mathsf{hs}$ is made of continuous (essential) spectrum
$\{\lambda \in \mathbb{R}\,;\, \lambda \leq -\nu_0\}$ where
$\nu_0=\inf_{\v \in \R^3}\sigma_{\mathsf{hs}}(\v)
>0$ and a  decreasing sequence of real eigenvalues with finite algebraic
multiplicities which unique possible cluster point is  $-\nu_0$.
Then, since $0$ is an eigenvalue of $\mathcal{L}_\mathsf{hs}$
associated to $\mathcal{M}$, one sees from the above Proposition
that the spectral gap $\mu_\mathsf{hs}$ of $\mathcal{L}_\mathsf{hs}$
satisfies
\begin{equation*}
\mu_\mathsf{hs}:=\min\bigg\{\lambda  :  -\lambda \in (-\nu_0,0), - \lambda \in
\mathfrak{S}(\mathcal{L}_\mathsf{hs})\setminus \{0\}\bigg\}\geq C^\star
\mu_\mathsf{max} \geq \frac{\eta\min\{\kappa,2\kappa(1-\kappa)\}}{\sqrt{5}}
.\end{equation*}
\end{nb}

To summarize, one gets the following coercivity estimate for the
Dirichlet form:
\begin{theo}\label{coercivity} For $\Q=\mathcal{L}_\mathsf{hs}$ or $\mathcal{L}_\mathsf{max}$, one has the
following:
$$-\int_{\R^3} \Q(f)(\v)f(\v)\mathcal{M}^{-1}(\v)\d\v
\geq \ap \|f-\varrho_f \mathcal{M}\|_{L^2(\mathcal{M}^{-1})}^2,
\qquad  \forall f \in \mathscr{D}(\Q)$$ where,
$\varrho_f=\displaystyle \int_{\R^3}f(\v)\d\v,$ and
$\ap=\mu_\mathsf{max}$ whenever  $\Q=\mathcal{L}_{\mathsf{max}}$
while, for hard-spheres interactions, {\em i.e.},
$\Q=\mathcal{L}_{\mathsf{hs}}$, one has $\ap \geq C^\star
\mu_\mathsf{max}$.
\end{theo}
\begin{proof} If $\varrho_f=0$, the proof follows directly from
Proposition \ref{linkHM} and \eqref{nullspace}. Now, if $f$ is a
given function with non-zero mean $\varrho_f$, set
$h=f-\varrho_f\mathcal{M}$. Then, $\varrho_h=0$ so that
$$-\int_{\R^3} \Q (h)(\v)h(\v)\mathcal{M}^{-1}(\v)\d\v \geq
\mu \|h\|_{L^2(\mathcal{M}^{-1})}^2.$$ This leads to the result
since $\Q(h)=\Q(f)$ and $\ds \int_{\R^3} \Q(f) \d\v=0$.
\end{proof}

Adopting now the entropy language, one obtains the following
relaxation rate, which is also new in the context of linear
Boltzmann equation:

\begin{cor}\label{dechomogene}
Let $f_0(\v) \in L^2(\R_\v^3,\mathcal{M}^{-1}(\v)\d\v)$ be
given and let $f(t)$ be the unique solution of \eqref{evol} with
initial condition $f(t=0)=f_0.$ Then, for any $t \geq 0,$ one has
the following
$$\left\|f(t)-\mathcal{M}\right\|_{L^2(\mathcal{M}^{-1})} \leq \exp(-\ap  t)
\left\|f_0-\mathcal{M}\right\|_{L^2(\mathcal{M}^{-1})}, \qquad
\forall t \geq 0,$$ where $\ap=\mu_\mathsf{max}$ when
$\Q=\Q_{\mathsf{max}}$ while, for hard-spheres interactions, {\em i.e.},
$\Q=\Q_{\mathsf{hs}}$, one has $\ap \geq C^\star \mu_\mathsf{max}$.
\end{cor}

We state another corollary of the above Theorem \ref{coercivity} in
which we strengthen the coercivity estimate:
\begin{cor}\label{coercivity2}  For $\Q=\mathcal{L}_\mathsf{hs}$ or $\mathcal{L}_\mathsf{max}$, there exists $c_\sigma >0$ such
that
$$-\int_{\R^3} \Q(f)(\v)f(\v)\mathcal{M}^{-1}(\v)\d\v
 \geq c_\sigma \|(f-\varrho_f
\mathcal{M}) \sqrt{\sigma} \|_{L^2(\mathcal{M}^{-1})}^2 \qquad
\qquad \forall f \in \mathscr{D}(\Q)$$ where,
$\varrho_f=\displaystyle \int_{\R^3}f(\v)\d\v$ and $\sigma(v)$ is
the collision frequency associated to $\Q$.
\end{cor}
\begin{proof} If $\Q=\mathcal{L}_\mathsf{max}$, since $\sigma_{\mathsf{max}}(v)=1$ the estimate is nothing
but Theorem \ref{coercivity}. Let us consider now the hard-spheres
case, $\Q=\mathcal{L}_\mathsf{hs}.$ Arguing as in the proof of
Theorem \ref{coercivity}, it suffices to prove the result for $f
\bot \, \mathcal{M}$, {\em i.e.}, whenever $\varrho_f=0$. We recall from
\cite{ArLo} that $\mathcal{L}_\mathsf{hs}$ splits as
$$\mathcal{L}_\mathsf{hs}f=\mathscr{K}f-\sigma_\mathsf{hs}
f, \quad  f \in \D(\mathcal{L}_\mathsf{hs})$$ where $\mathscr{K}$ is
a bounded (and compact) operator in $L^2(\mathcal{M}^{-1})$. We then have
\begin{equation*}\begin{split}
\|f\sqrt{\sigma_{\mathsf{hs}}}
\|^2_{L^2(\mathcal{M}^{-1})}&=\int_{\R^3_\v}\mathscr{K}(f)\,f \,
\mathcal{M}^{-1}\d \v -\int_{\R^3_\v}\mathcal{L}_\mathsf{hs}(f)\,f
\, \mathcal{M}^{-1}\d \v\\
&\leq \|\mathscr{K}\|
\|f\|^2_{L^2(\mathcal{M}^{-1})}+\D_\mathsf{hs}(f)\leq
\bigg(\frac{\|\mathscr{K}\|}{\mu_\mathsf{max}C^\star} +1
\bigg)\D_\mathsf{hs}(f)
 \end{split}\end{equation*}
where $\|\mathscr{K}\|$ stands for the norm of $\mathscr{K}$ as a
bounded operator on $L^2(\mathcal{M}^{-1})$ and we used Theorem
\ref{coercivity}. The corollary follows with
$$
c_\sigma=\frac{C^\star
\mu_\mathsf{max}}{\|\mathscr{K}\|+C^\star\mu_\mathsf{max}}.
$$
\end{proof}

\begin{nb}
Here again, as in Prop.~\ref{linkHM}, the constant $c_\sigma >0$ can
be quantitatively estimated using for instance the
estimate$\|\mathscr{K}\| \leq \frac{2\pi}{(1+\tau)^2}
\sqrt{\frac{\pi\Theta_1}{m_1}} $ that can be deduced without major
difficulty  from the explicit expression of $\mathscr{K}$ provided
in \cite{ArLo} with $\tau=(1-2\kappa)/\kappa >0$.
\end{nb}

\begin{nb} Recalling that $\sigma_\mathsf{hs}$ behaves like
$(1+|v|)$, the above corollary allows to control from below the
entropy production functional by the weighted
$L^2((1+|v|)\mathcal{M}^{-1},\d\v)$ norm. Such a weighted estimate
shall be very useful for the diffusion approximation.
\end{nb}

\section{Diffusion Approximation}\label{sec:diff}
\setcounter{equation}{0}

We shall assume again in this whole section that $u_1=0$. From the results of
the previous section, it is possible to derive some exact
convergence results for the solution of the re-scaled linear kinetic
Boltzmann equation
\begin{equation}\label{boltrescaled} \varepsilon \partial_t
f_\varepsilon(t,\x,\v)+v \cdot \nabla_\x
f_\varepsilon(t,\x,\v)=\dfrac{1}{\varepsilon}\Q(f_\varepsilon)(t,\x,\v),\end{equation}
with initial condition $f_\varepsilon(\x,\v,0)=f_0(\x,\v) \geq0$,
with $(x,v) \in \R^3  \times \R^3$. Note that all the analysis we
perform here is also valid if the  spatial domain denotes  the
three-dimensional torus $\mathbb{T}^3_\x$. One shall prove that
$f_\var$ converges, as $\var \to 0$, to $\mathcal{M}(v)\varrho$
where $\varrho=\varrho(t,\x)$ is the solution to the (parabolic)
diffusion equation:
\begin{equation}\label{diff}\begin{cases}
\partial_t \varrho(t,\x)  =
\nabla_x \cdot \left(\mathsf{D} \nabla_\x \varrho(t,\x) + u_1 \, \rho \right),\qquad \qquad
t \geq 0, \:\x \in \R^3,\\
\varrho(0,x) =\varrho_0(\x)=\ds\int_{\R_\v^3}f_0(\x,\v)\d\v
\end{cases}
\end{equation}
where the diffusion coefficient  $\mathsf{D}$ depends on the model
we investigate (hard-sphere interactions or Maxwell molecules). One
shall adopt here the strategy of \cite{DeGoPo,LiTo}. Namely, to
prove the convergence of the solution to \eqref{boltrescaled}
towards the solution $\varrho$ of \eqref{diff}, the idea is to use
the {\it a priori} estimate given by the production of entropy, as
in~\cite{LiTo} where this idea was applied to discrete velocity
models of the Boltzmann equation. Let us define the number density
and the current vector
  \[ \varrho_\var(t,\x) = \int_{\R_\v^3} f_\var(t,\x,\v)
  \d\v, \qquad  j_\var(t,\x) = \frac{1}{\var} \, \int_{\R_\v^3} f_\var(t,\x,\v) \, v \,  \, \d\v. \]
We also define $h_\var$ as
  \[ h_\var (t,\x,\v) =\dfrac{1}{\var}\bigg( f_\var(t,\x,\v) - \varrho_\var(t,\x)\mathcal{M}(\v) \bigg).\]
Integrating \eqref{boltrescaled} with respect to $\x$ and $\v$ and
using the fact that the mean of $\Q(f_\var)$ is zero, one gets the
mass conservation identity \begin{equation}\label{mass}
\int_{\R^3_\x \times \R_\v^3} f_\var(\x,\v,t)\d\x\d\v=\int_{\R^3_\x
\times \R_\v^3}f_0(\x,\v)\d\x\d\v,\end{equation} which means (using the
fact that the equation preserves non-negativity) that,
for any $T
>0$, the sequence $\varrho_\var(\x,t)$ is bounded in
$L^{\infty}(0,T;L^1(\R^3_\x)).$ Now, multiplying
\eqref{boltrescaled} by $f_\var\mathcal{M}^{-1}$ and integrating
over $\R^3_\x \times \R_\v^3$, we get
\begin{multline}\label{Apriori1}
\dfrac{1}{2}\dfrac{\d}{\d t}\int_{\R^3_\x \times \R_\v^3}
f^{\,2}_\var(t,\x,\v) \mathcal{M}^{-1}(\v)\d\x\d\v
+\dfrac{1}{2\var}\int_{\R^3_\x \times
\R_\v^3}\mathrm{div}_\x\left(\v
f^{\,2}_\var(t,\x,\v)\right)\mathcal{M}^{-1}(\v)\d\x\d\v \\
-\dfrac{1}{\var^2}\int_{\R^3_\x \times \R_\v^3}f_\var\, \Q (f_\var)
\mathcal{M}^{-1}\d\x\d\v=0.
\end{multline}
Now, because of the divergence form of the integrand, one sees that
the second term in \eqref{Apriori1} is zero while, because of
Corollary \ref{coercivity2},
\begin{equation}\begin{split}\label{Apriori2}
-&\dfrac{1}{\var^2}\int_{\R^3_\x \times \R_\v^3}f_\var\, \Q (f_\var)
\mathcal{M}^{-1}\d\x\d\v \\
 &\geq \dfrac{c_\sigma}{\var^2}\int_{\R^3_\x}
\|f_\var(t,x,\cdot)-\varrho_\var(t,x) \mathcal{M}\|_{L^2(\R_v^3,\sigma(v)
\mathcal{M}^{-1}(v)\d v)}^2\d x\\
&= c_\sigma \int_{\R^3_\x \times
\R_\v^3}h^{\,2}_\var(t,\x,\v)\mathcal{M}^{-1}(\v)\sigma(v)\d\x\d\v.\end{split}
\end{equation}
Consequently, Eq. \eqref{Apriori1}, together with \eqref{Apriori2},
leads to
$$\dfrac{1}{2}\dfrac{\d}{\d t}\int_{\R^3_\x \times \R_\v^3}
f^2_\var(\x,\v,t)\mathcal{M}^{-1}(\v)\d\x\d\v \leq
-c_\sigma\int_{\R^3_\x \times \R_\v^3}h^2_\var(\x,\v,t)\sigma(v)
\mathcal{M}^{-1}(\v)\d\x\d\v.$$ Defining therefore the following
Hilbert space:
$$ \mathcal{H}=L^2(\R^3_\x \times \R_v^3,\mathcal{M}^{-1}(\v)\d\x\d \v)$$ endowed
with its natural norm $\|\cdot\|_{\mathcal{H}}$, one has
\begin{equation}\label{estim}
\left\|f_\var(t)\right\|_{\mathcal{H}}^2 + 2c_\sigma \int_0^t
\left\|h_\var(s)\sqrt{\sigma}\right\|_{\mathcal{H}}^2\d s \leq
 \left\|f_0\right\|_{\mathcal{H}}^2, \qquad \forall t \geq
 0.\end{equation}
We obtain the following  {\it a priori} bounds:
\begin{propo}\label{apriori} For any $\var >0$, let $f_\var(t)$ denotes the unique
solution to \eqref{boltrescaled} with $f_0 \in \mathcal{H}$, $f_0
\geq 0$. 
 Then, for any $0 \leq T
< \infty$
\begin{enumerate}
\item The sequence $(f_\var)_\var$ is bounded in
$L^{\infty}\left(0,T\,;\,\mathcal{H}\right),$
\item the sequence $(\sqrt{\sigma}h_\var)_\var$ is bounded in $L^2\left(0,T;\mathcal{H}\right),$
\item the density sequence $(\varrho_\var)_\var$ is bounded in
$L^{\infty}(0,T\,;\,L^1(\R^3_\x) \cap L^2(\R^3_\x)),$
\item the current sequence $(j_\var)_\var$ is bounded in $\left[L^2((0,T)
\times \R^3_\x)\right]^3.$
\end{enumerate}
\end{propo}
\begin{proof} The first two points are direct consequences of
\eqref{estim} with
$$\sup_{\var
>0}\left\|f_\var\right\|_{L^{\infty}\left(0,T\,;\,\mathcal{H}\right)}
\leq \left\|f_0\right\|_{\x,\v}, \qquad \sup_{\var
>0}\left\|\sqrt{\sigma}h_\var\right\|_{L^2\left(0,T\,;\,\mathcal{H}\right)} \leq
 (2c_\sigma)^{-1/2}\left\|f_0\right\|_{\x,\v}.$$
Now, Eq. \eqref{mass} proves that the number density sequence
$(\varrho_\var)_\var$ is bounded in
$L^{\infty}(0,T\,;\,L^1(\R^3_\x))$ and, according to Cauchy-Schwarz
inequality,
$$0 \leq \varrho_\var(t,\x) \leq \bigg(\int_{\R_\v^3}
f^{\,2}_\var(t,\x,\v) \mathcal{M}^{-1}(\v)\d \v\bigg)^{1/2}$$ we see
from point \textit{(1)} that $(\varrho_\var)_\var$ is also bounded
in $L^{\infty}(0,T\,;\,L^2(\R^3_\x))$. Finally, since
$f_\var=\varrho_\var \mathcal{M} + \var h_\var$ and $\int_{\R_\v^3}
\v  \mathcal{M}(\v)\d \v= 0$, one has
\begin{equation*}
\int_0^T \d t \int_{\R^3_\x} \left|j_\var(t,\x)\right|^2 \d\x =
\int_0^T\d t\int_{\R^3_\x}\d \x \left|\int_{\R_\v^3} v
h_\var(t,\x,\v)\d\v\right|^2
 \end{equation*}
while, from Cauchy-Schwarz inequality and the fact that $\sigma$ is
bounded from below
\begin{equation*}
 \left|\int_{\R_\v^3} v
h_\var(t,\x,\v)\d\v\right|^2  \leq \bigg(\int_{\R_\v^3}|\v
|^2\mathcal{M}(\v) \d\v\bigg)\bigg(\int_{\R_\v^3}h^{\,2}_\var
\mathcal{M}^{-1} \d\v\bigg)
\end{equation*}
so that
$$\int_0^T \d t \int_{\R^3_\x} \left|j_\var(t,\x)\right|^2 \d\x \leq
\frac{3\Theta^{\#}}{m}\int_0^T \left\|h_\var(t)\right\|_{\x,\v}^2 \d
t$$ and the conclusion follows from point \textit{(2)}.\end{proof}
\begin{nb}
 Since $f_\var=\var h_\var + \varrho_\var \mathcal{M},$ noticing that $\int_{\R^3_\v}\sigma(v) \mathcal{M}(v) \d v < \infty$, one
deduces from the above points \textit{(2)} and \textit{(3)} and that
the sequence $(\sqrt{\sigma}f_\var)_\var$ is bounded in
$L^2\left(0,T;\mathcal{H}\right).$
\end{nb}
For any $T >0$, we define
$$\O_T=(0,T)
\times \R^3_\v \times \R^3_x \quad \text{ and } \quad \d\mu_T=\d x
\d v \d t.$$ The bounds provided by Prop.~\ref{apriori} allows
to assume that, up to a subsequence,
$$f_\var \rightharpoonup f \quad \text{ in } \quad  L^2(\O_T\,;\,\sigma \mathcal{M}^{-1}\d\mu_T), \qquad \qquad
h_\var \rightharpoonup h \quad \text{ in } \quad L^2(\O_T\,;\,\sigma
\mathcal{M}^{-1}\d\mu_T);$$
$$\varrho_\var \rightharpoonup \varrho  \quad \text{ in } \quad L^2((0,T) \times
\R^3_\x), \qquad  \qquad   j_\var \rightharpoonup j \quad \text{ in
} \quad  \left[L^2((0,T) \times \R^3_\x)\right]^3.$$ Let $\Psi \in
 L^2(\O_T,\sigma^{-1}
\mathcal{M}\d\mu_T)=\left[L^2(\O_T,\sigma
\mathcal{M}^{-1}\d\mu_T)\right]^\star$ be given. Since
$\sigma=\sigma_\mathsf{max}$ is constant while
$\sigma=\sigma_\mathsf{hs}$ behaves asymptotically like $(1+|\v|)$,
one easily has from Cauchy-Schwarz
$$\varphi(t,x)=\int_{\R_\v^3} \mathcal{M}(v)\Psi(t,x,v)
\d \v \in L^2((0,T) \times \R^3_x),$$
%
and therefore
$$\lim_{\var \to 0}\int_0^T \d t \int_{\R^3_x} \varrho_\var
(t,x)\varphi(t,x)\d x=\int_0^T \d t \int_{\R^3_x} \varrho
(t,x)\varphi(t,x)\d x.$$
Thus, writing $f_\var=\varrho_\var
\mathcal{M} + \var h_\var$, one checks that
$$\lim_{\var \to 0}\int_{\O_T}f_\var \Psi \d\mu_T =\int_0^T \d t \int_{\R^3_x}
\varrho (t,x)\varphi(t,x)\d x= \int_{\O_T} \varrho \mathcal{M} \Psi
\d\mu_T,$$ {\em i.e.}, $f_\var \rightharpoonup \varrho \mathcal{M}$ in
$L^2(\O_T,\sigma \mathcal{M}^{-1}\d\mu_T)$. In particular,
$f(t,x,v)=\varrho(t,x)\mathcal{M}(v).$ Moreover,
\begin{equation}\label{hep} \lim_{\var
\to 0} \int_{\O_T} h_\var \Psi \d\mu_T = \int_{\O_T} h \Psi
\d\mu_T.\end{equation} for any $\Psi=\Psi(t,x,v) \in L^2(\O_T,
\sigma^{-1}\mathcal{M}\d\mu_T)$. Now, choosing $\Psi$ independent of
$v$, one sees that
$$\int_{\R^3_\v} h(t,x,v)\d v=0, \qquad \forall t >0, \quad \x \in
\R^3_\x.$$ Finally, using in \eqref{hep} a test function
$\Psi(t,x,v)=v\varphi(t,x)$ with $\varphi \in L^2((0,T) \times
\R^3_\x)$, we deduces from the weak convergence of $j_\var$ to $j$
that
$$j(t,x)=\int_{\R^3_\v} v h(t,x,v) \d \v.$$
Finally, integrating equation \eqref{boltrescaled} over $\R_\v^3$
leads to the continuity equation
\begin{equation}\label{continuvar}
 \partial_t \varrho_\var(t,x)
+ \mathrm{div}_\x  j_\var(t,x)  =0, \qquad \qquad \forall \var
>0.\end{equation}
We deduce at the limit that
\begin{equation}\label{continu} \partial_t \varrho(t,x) +  \mathrm{div}_\x  j(t,x)=0, \qquad \qquad t>0, \quad x \in \mathbb{T}_x^3\end{equation}
in the \textit{distributional sense}. We summarize these first
results in the following:
\begin{propo}\label{weakconv} Under the assumptions of Proposition \ref{apriori}, for any $T >0$, up
to a subsequence,
\begin{enumerate}[i)]
\item $(\varrho_\var)$ converges weakly in $L^2((0,T) \times
\R^3_\x)$ to some $\varrho;$
\item $(h_\var)$ converges weakly in
$L^2(\O_T,\sigma\mathcal{M}^{-1}\d\mu_T)$ to some function $h$ with
$\null\qquad\qquad\qquad$  $\ds \int_{\R^3_\v} h(t,x,v)\d v=0$;
\item $(f_\var)$ converges weakly to $\varrho \mathcal{M}$ in
$L^2(\O_T,\sigma\mathcal{M}^{-1}\d\mu_T)$;
\item $(j_\var)$ converges weakly to $j(t,x)=\ds \int_{\R^3_\v}v h(t,x,v) \d v$ in $\left[L^2((0,T) \times
\R^3_\x)\right]^3.$
\end{enumerate}
where $\varrho$ and $j$ are related by \eqref{continu}.
\end{propo}

The problem of the diffusion approximation is then reduced to the
one of finding a suitable relation, similar to the\textit{ classical
Fick's law}, linking the current $j(t,x)$  to the gradient of the
density $\varrho(t,x)$. Such a Fick's law (and the corresponding
coefficient) shall depend heavily on the collision kernel.

\subsection{Maxwell molecules}\label{max}

When dealing with Maxwell molecules, {\em i.e.}, whenever
$\Q=\mathcal{L}_\mathsf{max}$, it is possible to obtain an explicit
expression for the diffusion coefficient. Precisely, multiplying
equation~\eqref{boltrescaled} by $v$ and integrating over $\R_\v^3$
gives
\begin{equation}\label{current} \var^2 \,
\partial_t j_\var(t,x)
     + \left( \int_{\R_\v^3} (v \otimes \v) : \nabla_\x f_\var(t,x,v)\d v \right)
     = \frac{1}{\var} \, \int_{\R_\v^3} \mathcal{L}_{\mathsf{max}} (f_\var) \, v  \d\v
     \end{equation}
Now, as we already saw it (see \eqref{integralv}): $$\int_{\R_\v^3}
\mathcal{L}_{\mathsf{max}} (f_\var) \,v\d\v =-\alpha (1-\beta) \var
j_\var=-\lambda_{0,1}\var j_\var.$$ Then, recalling that
$f_\var(t,x,v)=\varrho_\var(t,x)\mathcal{M}(v)+\var h_\var(t,x,v)$,
Eq. \eqref{current} becomes
\begin{equation}\label{current1}
\var^2 \,
\partial_t j_\var(t,x)
     + \mathbf{A} : \nabla_\x \varrho_\var(t,x) +
     \var \left( \int_{\R_\v^3} (v \otimes \v) : \nabla_\x h_\var(t,x,v)\d v \right)
     =- {\lambda}_{0,1}\: j_\var \end{equation}
     where $\mathbf{A}$ is the matrix of directional temperatures associated to the
  distribution $\mathcal{M}$:
$$ \mathbf{A}  = \int_{\R_\v^3} (v\otimes \v)\,
\mathcal{M}(\v) \, \d\v=\dfrac{\Theta^{\#}}{m}
\mathbf{Id}=\mathrm{diag}\left(\dfrac{\Theta^{\#}}{m}\,;\dfrac{\Theta^{\#}}{m}\,;\dfrac{\Theta^{\#}}{m}\right).$$
One may rewrite \eqref{current1} as
$$\dfrac{\Theta^{\#}}{m} \nabla_\x \varrho_\var(t,x) +
\lambda_{0,1} j_\var(t,x)=-\var^2 \,
\partial_t j_\var(t,x)- \var \left( \int_{\R_\v^3} (v \otimes \v) : \nabla_\x h_\var(t,x,v)\d v
\right).$$ Choosing a test-function $\psi \in
\mathscr{C}^\infty _c((0,T) \times \R^3_x)$, the above equation reads
in its distributional form:
\begin{multline*}
\dfrac{\Theta^{\#}}{m} \int_0^T \d t \int_{\R^3_\x}\nabla_\x
\psi(t,x) \varrho_\var(t,x) \d x - \lambda_{0,1} \int_0^T \d t
\int_{\R^3_\x} \psi(t,x)
j_\var(t,x)\d t=\\
-\var^2 \, \int_0^T \d t \int_{\R^3_\x}
\partial_t \psi(t,x) j_\var(t,x)\d x - \var \int_{\O_T}  h_\var(t,x,v)  (v \otimes \v)
:  \nabla_\x \psi(t,x) \d \mu_T
\end{multline*}
and, by virtue of the bounds in Prop.~\ref{apriori}, the right-hand side
converges to zero as $\var \to 0$ and  one gets at the limit:
\begin{equation}
j (t,x)=-\dfrac{\Theta^{\#}}{m\lambda_{0,1}} \nabla_\x
\varrho(t,x)\end{equation} in the distributional sense. \textit{The
above formula provides the so-called Fick's law for Maxwell's
molecules}. One deduces the following Theorem:
\begin{theo}\label{limdiff}
Let $f_0 \in \mathcal{H}$ and, for any $\var >0$, let
$f_\var(t,x,v)$ denotes the solution to \eqref{boltrescaled}. Then,
for any $T >0$, up to a sequence, $f_\var$ converges strongly in
$L^2_\mathrm{loc}(\O_T\,;\,\mathcal{M}^{-1}\d\mu_T)$ towards
$\varrho(t,x)\mathcal{M}(v)$,
where $\varrho(t,\x)$ is the solution to  the diffusion equation
    \begin{equation}\label{heat}
    \partial_t \varrho =\nabla_\x \cdot \left(\dfrac{\Theta^{\#}}{m\lambda_{0,1}}
\nabla_\x \varrho(t,x)\right),\qquad \varrho(t=0,x)=\ds
\int_{\R_v^3}f_0(\x,v)\d v.
    \end{equation}
 \end{theo}

\begin{proof} We already proved that $f_\var$ converges weakly to $\varrho
\mathcal{M}$ in $L^2((0,T)\,;\, \mathcal{H})$. To prove the strong
convergence, since
$$\int_0^T  \|f_\var(t)-\varrho_\var(t,x)
\mathcal{M}\|_{\mathcal{H}}^2=\var^2 \int_0^T
\|h_\var(t)\|^2_\mathcal{H} \to 0$$ it suffices to prove that
$\varrho_\var \mathcal{M} $ converges strongly to $\varrho
\mathcal{M}$ in $\mathcal{H}$. This is equivalent to prove that
$\varrho_\var $ converges strongly to $\varrho$ in $L^2(
0,T\,;\,L^2_\mathrm{loc}(\R^3_\x))$. This is done in the spirit of
\cite{LiTo} and \cite{DeGoPo} by using a compensated-compactness
argument. Precisely, let us define the following  vectors of $
\R^3_\x \times \R^+_t:$ $\mathbf{U}_\var=(j_\var, \varrho_\var)$ and
$ \mathbf{V}_\var=(0,\varrho_\var).$ From \eqref{continuvar}, one
sees that $\mathrm{div}_{x,t}\mathbf{U}_\var=0$, in particular
$(\mathrm{div}_{x,t}\mathbf{U}_\var)_\var$ is bounded in
$L^2(\R^3_\x \times \R^+_t)$. Now, from \eqref{current1}, one sees
that $\mathbf{A} : \nabla_x \varrho_\var$ is a bounded family in
$L^2((0,T) \times \R^3_x)$. Since $\mathbf{A}=\frac{\Theta^{\#}}{m}
\mathbf{Id}$, it is clear that
$$\mathrm{curl} \,\mathbf{V}_\var=\left(%
\begin{array}{cc}
  0 & -^T\nabla_\x \varrho_\var \\
  \nabla_x \varrho_\var & 0 \\
\end{array}%
\right)$$ is bounded in $[L^2_\mathrm{loc}((0,T) \times \R^3_x)]^{4
\times 4}$. Now, from the div-curl Lemma \cite{Mu,Ta},
$\mathbf{U}_\var \cdot \mathbf{V}_\var=\varrho^2_\var$ converges to
$\varrho^2$ in $\mathcal{D}'((0,T) \times \R^3_x)$.

Moreover, we already saw that $\varrho_\var$ is bounded in $L^\infty( 0,T
\,;\,L^2(\R^3_\x))$ from which we deduce the strong convergence of
$\varrho_\var$ to $\varrho$ in
$L^2((0,T)\,;\,L^2_\mathrm{loc}(\R^3_\x))$.\end{proof}

\begin{nb} \label{nb:FickMax}
As already pointed out in \cite{SpTo}, the dependence of the
diffusivity
$\mathsf{D}_\mathsf{max}:={\Theta^{\#}}/{m\lambda_{0,1}}$ on the
inelasticity parameter $\beta$ shows that inelasticity tends to slow
down the diffusive process.
\end{nb}

\subsection{Hard spheres} When dealing with hard-spheres
interactions, it appears difficult to obtain an explicit expression
of the diffusion coefficient. Nevertheless, its existence can be
deduced from Theorem \ref{coercivity}. Indeed,  a direct consequence
of the Fredholm Alternative is the  following:

\begin{propo}\label{chi} For any  $\null \ i=1,2,3$, the equation
$$\mathcal{L}_\mathsf{hs}(\chi_i)= v_i  \mathcal{M}(\v),\qquad \qquad v \in
\R^3$$
 has a unique solution $\chi_i \in
L^2(\sigma(\v)\mathcal{M}^{-1}(\v)\d\v)$, such that  $\langle \chi_i,
\mathcal{M} \rangle = \ds \int_{\R_\v^3}\chi_i(v)d v=0$ for any $\null \
i=1,2,3.$
\end{propo}

\begin{nb}
  Note that the above Proposition holds true only because we assumed
  the bulk velocity $u_1$ to be zero, {\em i.e.}, $\int_{\R^3_\v} \v \mathcal{M}\d\v=0$.
  If one deals with a non-zero
  bulk velocity $u_1$, then
  if one denotes 
  $a(v)=v-u_1$, $\chi_i$ then solves $\mathcal{L}_\mathsf{hs}(\chi_i)= a_i(v)
  \mathcal{M}(\v)$ (see also \eqref{a(v)}), and moreover the limit diffusion equation
  includes in this case an additional drift term $u_1 \cdot \nabla_x \rho$, see Eq.~\eqref{diff}.
\end{nb}

Then, setting $\chi=(\chi_1,\chi_2,\chi_2)$ one defines the
diffusion matrix:
$$\mathbf{D}:=-\int_{\R_\v^3} v \otimes \chi(\v) \d \v \in \R^{3
\times 3}.$$ Adapting the result of \cite{Po}, the diffusion matrix
is given by $\mathbf{D}=
\mathrm{diag}(\mathsf{D_\mathsf{hs},D_\mathsf{hs},D_\mathsf{hs}})$
for some {\it positive constant} $\mathsf{D}_\mathsf{hs} >0$,
namely,
$$\mathsf{D}_\mathsf{hs}=-\int_{\R_\v^3} v_1 \chi_1(\v)\d v=-\int_{\R_\v^3}
\mathcal{L}_\mathsf{hs}(\chi_1)\chi_1 \mathcal{M}^{-1} \d\v \geq \mu
\|\chi_1\|^2_{L^2(\mathcal{M}^{-1})}.$$

\begin{nb}
Note that, when dealing with Maxwell molecules, for any $i=1,2,3,$
the function $\chi_i$ appearing in Proposition \ref{chi} is given by
$\chi_i=- \frac{1}{\lambda_{0,1}} \v_i \mathcal{M} $ and we find
again the expression of the diffusion matrix $ \mathbf{D}=
\frac{\Theta^{\#}}{m\lambda_{0,1}}\mathbf{Id}.$
\end{nb}

%
Recall that, for any $T >0$, we defined $\O_T=(0,T) \times \R^3_\v
\times \R^3_x$ and $\d\mu_T=\d x \d v \d t.$ Then, for any $\phi \in
L^\infty(\R^3_\v)$ and any $\psi \in \mathscr{C}_c ^\infty((0,T)
\times \R^3_x)$, multiplying Eq. \eqref{boltrescaled} by
$\phi(v)\psi(t,x)$ and integrating over $\O_T$ one has
\begin{multline}\label{refe}
\int_{\O_T} \varrho_\var(t,x)\mathcal{M}(v) \left(v \cdot \nabla_x
\psi(t,x)\right) \phi(v)\d\mu_T + \int_{\O_T}
\mathcal{L}_\mathsf{hs}(h_\var) \phi(v)\psi(t,x)\d\mu_T=\\
\var\bigg(\int_{\O_T} \phi\,f_\var \partial_t \psi\, \d\mu_T +
\int_{\O_T} h_\var  (v \cdot \nabla_x \psi)\phi \d\mu_T\bigg).
\end{multline}
In particular, by virtue of Propositions \ref{apriori} and
\ref{weakconv}, one sees that
$$\lim_{\var \to 0} \bigg( \int_{\O_T} \mathcal{M}(v)  \varrho_\var(t,x) \left(v \cdot \nabla_x
\psi(t,x)\right) \phi(v)\d\mu_T + \int_{\O_T}
\mathcal{L}_\mathsf{hs}(h_\var) \phi(v)\psi(t,x)\d\mu_T \bigg)=0.$$
Now, one
deduces easily as in \cite{DeGoPo} that
\begin{equation}\label{a(v)}
\int_{\O_T} \varrho (t,x) \left(v \cdot \nabla_x
\psi(t,x)\right)\mathcal{M}(v) \phi(v)\d\mu_T=-\int_{\O_T}
\mathcal{L}_\mathsf{hs}(h ) \phi(v)\psi(t,x)\d\mu_T,
\end{equation}
which means that, in the distributional sense,
$$\mathrm{div}_\x
(v\mathcal{M}(v) \varrho(t,x))=\mathcal{L}_\mathsf{hs}(h), \qquad t
>0, x \in \R^3_x.$$
 Since $h$ is of zero $\R^3_v$--average, Proposition \ref{chi} asserts
that
$$h(t,x,v)=-\chi(v) \cdot \nabla_\x \varrho(t,x)$$
and Proposition \ref{weakconv} \textit{(iv)} leads to
$$j(t,x)=\int_{\R^3_\v} v h(t,x,v)\d v = \mathbf{D} : \nabla_\x
\varrho(t,x).$$
We then obtain the following:
\begin{theo} \label{FickHS} Let $0 \leq f_0(x,v) \in L^2(\R^3_\x \times
\R^3_v,\mathcal{M}^{-1}\d\v)$ be given and let $f_\var$ be the
associated sequence of solution to \eqref{boltrescaled} where
$\Q=\Q_\mathsf{hs}$. Then, up to a subsequence, $f_\var $ converges
strongly in $L^2_{\mathrm{loc}}(\O_T,\mathcal{M}^{-1}\d\mu_T)$ to
$\varrho(t,x)\mathcal{M}$ where $\varrho \geq0$ is the solution to
the parabolic  diffusion equation \eqref{diff} where the diffusion
coefficient $\mathsf{D}_\mathsf{hs}$ is given by
$$\mathsf{D}_\mathsf{hs}:=-\int_{\R_\v^3} \v_1 \chi_1(\v) \d \v \in \R^{3
\times 3}$$ with $ \chi_1 $   defined in Prop. \ref{chi}.
\end{theo}
\begin{proof} We already proved that $f_\var$ converges weakly to $\varrho
\mathcal{M}$ in $L^2((0,T), \mathcal{H})$ and the strategy to prove
the strong convergence is that used in Theorem \ref{limdiff}.
Precisely, we define again $\mathbf{U}_\var=(j_\var, \varrho_\var)$
and $ \mathbf{V}_\var=(0,\varrho_\var) $ and observes that again
$(\mathrm{div}_{x,t}\mathbf{U}_\var)_\var$ is bounded in
$L^2(\R^3_\x \times \R^+_t)$. Now, from \eqref{refe}, with
$\phi(v)=\dfrac{v}{|\v|}$ and setting $\Gamma=\ds \int_{\R^3_v}
\dfrac{v \otimes v}{|\v|} \mathcal{M}(v)\d v,$ one sees that
$\varrho_\var$ satisfies:
$$\Gamma : \nabla_\x \varrho_\var= \int_{\R^3_v}
 \mathcal{L}_\mathsf{hs}(h_\var)\frac{ v}{|v|}\d\v
- \var\bigg(\partial_t \int_{\R^3_\v} \frac{\v}{|\v|} \,f_\var \d \v
+ \mathrm{div}_x \left[\int_{\R^3_\v} \dfrac{v \otimes v}{|\v|}
h_\var \d\v\right]\bigg)$$ so that $ {\Gamma}: \nabla_\x
\varrho_\var$ lies in a bounded subset of $L^2_\mathrm{loc}((0,T)
\times \R^3_\x)$. Since $\Gamma$ is invertible, one proceeds as in
the proof of Theorem \ref{limdiff} that $\varrho_\var$ converges
strongly to $\varrho$ in $L^2_\mathrm{loc}((0,T) \times
\R^3_\x)$.\end{proof}

As we saw it, the diffusivity $\mathsf{D}_\mathsf{hs}$ associated to
hard-spheres interactions is not explicitly computable, the solution
$\chi$ not being explicit. It is however possible to   obtain a
quantitative estimate of $\mathsf{D}_\mathsf{hs}$ in terms of known
quantities ({\it i.e.}, that do not involve $\chi_1$):
\begin{propo} \label{FickHSestim} One has the following estimate:
$$\dfrac{\Theta^{\#}}{c_\mathsf{hs} m}  \leq \mathsf{D}_\mathsf{hs} \leq
\dfrac{\Theta^{\#}}{\lambda_{0,1}C^\star m}$$ where $C^\star$ is the
constant provided by Prop. \ref{linkHM} and
 $c_\mathsf{hs}=-\dfrac{\langle
\mathcal{L}_\mathsf{hs}(v_1
\mathcal{M}),v_1\mathcal{M}\rangle}{\|v_1
\mathcal{M}\|^2_{L^2(\mathcal{M}^{-1})}}
>0.$
\end{propo}
\begin{proof} We begin with the lower bound of $\mathsf{D}_\mathsf{hs}$.
For any $s \in \R$, let
$$P(s)=\langle
\mathcal{L}_\mathsf{hs}(\chi_1+sv_1\mathcal{M}),\chi_1+sv_1\mathcal{M}\rangle.$$
Since $\mathcal{L}_\mathsf{hs}$ is negative, one has $P(s) \leq 0$
for any $s \in \R$. Moreover,
\begin{equation*}\begin{split}
P(s)&=\langle \mathcal{L}_\mathsf{hs}( \chi_1),\chi_1 \rangle + 2s
\langle \mathcal{L}_\mathsf{hs}(\chi_1),v_1\mathcal{M}\rangle + s^2
\langle \mathcal{L}_\mathsf{hs}(v_1
\mathcal{M}),v_1\mathcal{M}\rangle\\
&=-\mathsf{D}_\mathsf{hs} +2s \|v_1
\mathcal{M}\|^2_{L^2(\mathcal{M}^{-1})} + s^2\langle
\mathcal{L}_\mathsf{hs}(v_1 \mathcal{M}),v_1\mathcal{M}\rangle.
 \end{split}\end{equation*}
We get therefore that
$$\mathsf{D}_\mathsf{hs} \geq 2s\|v_1
\mathcal{M}\|^2_{L^2(\mathcal{M}^{-1})} + s^2 \langle
\mathcal{L}_\mathsf{hs}(v_1 \mathcal{M}),v_1\mathcal{M}\rangle,
\qquad \forall s \in \R.$$ With the definition of $c_\mathsf{hs}$
(note that $c_\mathsf{hs} >0$ since $\mathcal{L}_\mathsf{hs}$ is
negative and $v_1\mathcal{M} \bot \mathcal{M}$), we get
$$\mathsf{D}_\mathsf{hs} \geq \left(2s-c_\mathsf{hs}s^2\right)\|v_1
\mathcal{M}\|^2_{L^2(\mathcal{M}^{-1})}, \qquad \forall s \in \R.$$
Optimizing with respect to $s$, one sees that
$$\mathsf{D}_\mathsf{hs} \geq \dfrac{1}{c_\mathsf{hs}}\|v_1
\mathcal{M}\|^2_{L^2(\mathcal{M}^{-1})}=
\dfrac{\Theta^{\#}}{c_\mathsf{hs}m}.$$ To get an upper bound for
$\mathsf{D}_\mathsf{hs}$, we use the fact that, thanks to
\eqref{integralv},
$$\mathsf{D}_\mathsf{hs}=-\langle \chi_1, v_1 \mathcal{M} \rangle
=\lambda_{0,1}^{-1} \langle \mathcal{L}_\mathsf{max}(\chi_1),v_1
\mathcal{M}\rangle.$$ Now, as above, for any $s \in \R$, define
$Q(s)=\langle \mathcal{L}_\mathsf{max}(s\chi_1 +v_1\mathcal{M}),s
\chi_1 + v_1 \mathcal{M}\rangle$. Here again, $Q(s) \leq 0$ for any
$s \in \R$ and
\begin{equation*}\begin{split}Q(s)&=s^2 \langle \mathcal{L}_\mathsf{max}(\chi_1),\chi_1\rangle +
2s \langle \mathcal{L}_\mathsf{max}(\chi_1),v_1 \mathcal{M}
\rangle+\langle \mathcal{L}_\mathsf{max}(v_1 \mathcal{M}),v_1
\mathcal{M}\rangle\\
&=s^2 \langle \mathcal{L}_\mathsf{max}(\chi_1),\chi_1 \rangle
+2\lambda_{0,1}\mathsf{D}_\mathsf{hs} s
-\lambda_{0,1}\|v_1\mathcal{M}\|^2_{L^2(\mathcal{M}^{-1})}.\end{split}\end{equation*}
Now, according to Prop. \ref{linkHM}, $\langle
\mathcal{L}_\mathsf{max}(\chi_1),\chi_1 \rangle \geq
\frac{1}{C^\star} \langle
\mathcal{L}_\mathsf{hs}(\chi_1),\chi_1\rangle=-\mathsf{D}_\mathsf{hs}/{C^\star}$
so that
$$ 0 \geq Q(s) \geq -\frac{\mathsf{D}_\mathsf{hs}}{C^\star}s^2+2\lambda_{0,1}\mathsf{D}_\mathsf{hs} s
-\lambda_{0,1}\|v_1\mathcal{M}\|^2_{L^2(\mathcal{M}^{-1})}, \qquad
\forall s \in \R.$$ Optimizing the right-hand side with respect to
$s \in \R$, we get
$$\lambda_{0,1}\|v_1\mathcal{M}\|^2_{L^2(\mathcal{M}^{-1})} \geq
\lambda_{0,1}^2 C^\star \mathsf{D}_\mathsf{hs}$$ which gives the
desired upper bound.
\end{proof}
\begin{nb} It is possible to provide some upper bound for
$c_\mathsf{hs}$. Namely, using the fact that there exists $\nu_1
>0$ such that $\sigma(v) \leq \nu_1 (1+|\v|)$, it is easy to see that

$$c_\mathsf{hs} \leq
\frac{1}{\|v_1\mathcal{M}\|^2_{L^2(\mathcal{M}^{-1})}} \int_{\R^3}
\sigma(v)^2 v_1^2\mathcal{M}(v) \d\v \leq
\frac{m\nu_1}{\Theta^{\#}}\int_{\R^3}(1+|v|)^4 \mathcal{M}(v)\d\v.$$
This very rough estimate could certainly be strengthened. Note also
that the upper bound for $\mathsf{D}_\mathsf{hs}$ reads as
$$\mathsf{D}_\mathsf{hs} \leq
\dfrac{\tau(1-\a)(1-\b)}{\a(1-\b)(1-\a(1-\b))}\dfrac{\sqrt{m_1\Theta_1}}{m}$$
where $\tau= \dfrac{\sqrt{5 }}{ \mathrm{erf}^{-1}(1/2)\sqrt{2}}
\simeq 3.3154$ is a numerical constant and we used the lower bound
of $C^\star$ provided by Remark \ref{nb:lower}.
\end{nb}

\section*{Appendix A}
\setcounter{equation}{0}\renewcommand{\theequation}{A.\arabic{equation}}

We provide here a constructive proof of the coefficient $\varrho_1$
appearing in Proposition \ref{linkHM} with the aim of finding
quantitative estimates for the coefficient $C^\star$ in Prop.
\ref{linkHM}. Namely, recalling that $\kappa=\a(1-\b)$, one has
\begin{lemmeA}\label{lem:append} Given $\varrho_0 >0$, there exists $ \varrho_1>0$ such
that
\begin{equation}\label{estimate1}
\int_{\{|\bar{w}-\bar{v}| \le \varrho_1\}}
\frac{\mathcal{M}_1(\bar{w})}{\left(|\bar{v}-\bar{w}|^2+
\xi^2\right)^{1/2}} \, \d \bar{w}
    \le \frac12 \,  \int_{\R^2} \frac{\mathcal{M}_1(\bar{w})}{\left(|\bar{v}-\bar{w}|^2+ \xi^2\right)^{1/2}} \, \d \bar{w}.
\end{equation}
  for  any $\bar{v} \in \R^2, \xi \in [0,\varrho_0/2\kappa]$.
  Moreover, setting
  $\eta=\sqrt{\frac{2\Theta_1}{m_1} }\mathrm{erf}^{-1}(\frac{1}{2}),$ where
  $\mathrm{erf}^{-1}$ is the inverse error function, one has
  $$\varrho_1 \geq
  \left(\eta^2-\frac{\varrho_0^2}{4\kappa^2}\right)^{1/2}, \qquad \forall 0 < \varrho_0 < 2\kappa \eta.$$
\end{lemmeA}
\begin{proof} We assume without loss of generality that $u_1=0$.
Let $\varrho_0 >0$ be given. Let us fix
$\bar{v}=(\bar{v}_1,\bar{v}_2) \in \mathbb{R}^2$ and $\xi \in
[0,\varrho_0/2\kappa]$. Using polar coordinates, it is clear that
\begin{multline*}
\left(\frac{m_1}{2 \pi \Theta_1} \right)^{-3/2} \, \int_{\{|\bar{w}-\bar{v}|
\le \varrho_1\}} \frac{\mathcal{M}_1(\bar{w})}{\left(|\bar{v}-\bar{w}|^2+
\xi^2\right)^{1/2}} \, \d \bar{w} = \\ \exp(-a |\bar{v}|^2)\int_0^{\varrho_1}
\dfrac{r\exp(-ar^2)}{\sqrt{r^2+\xi^2}}\d r \int_0^{2\pi} \exp\Big(-2ar(v_1\cos
\theta+v_2\sin \theta)\Big)\d \theta
 \end{multline*}
 where $a=m_1/(2\Theta_1)$.
Therefore,  a sufficient condition (independent of $\bar{v}$) for
\eqref{estimate1} to hold  is that
$$\int_0^{\varrho_1}
\dfrac{r\exp(-ar^2)}{\sqrt{r^2+\xi^2}}\d r \leq \dfrac{1}{2}
\int_0^\infty \dfrac{r\exp(-ar^2)}{\sqrt{r^2+\xi^2}}\d r, \qquad
\forall \xi \in [0,\varrho_0/2\kappa].$$ It is not difficult to see
that this is equivalent to
$$\mathrm{erf}\left(\sqrt{a(\varrho_1^2+\xi^2)}\right)- \mathrm{erf}(\sqrt{a}
\xi) \leq \dfrac{1}{2}-\dfrac{1}{2}\mathrm{erf}(\sqrt{a}\xi), \qquad
\forall \xi \in [0,\varrho_0/2\kappa]$$ where $\mathrm{erf}$ is the
error function $\mathrm{erf}(x)=\frac{2}{\sqrt{\pi}}\ds\int_0^x
\exp(-t^2)\d t,$ $x \geq 0.$ This allows to define a function:
$$z:\:\xi \in \R_+ \mapsto z(\xi)$$
where $z(\xi)$ is the nonnegative solution to  the identity
\begin{equation}\label{zxi}
\sqrt{a(z^2+\xi^2)}=\mathrm{erf}^{-1}\left(\frac{1}{2}+\frac{1}{2}\mathrm{erf}(\sqrt{a
}\xi)\right)\end{equation}
 where $\mathrm{erf}^{-1}$ is the inverse error
function. Clearly the Lemma is proven provided
$$\varrho_1:= \min\{ z(\xi),\;\xi \in [0,\varrho_0/2\kappa]\} >0.$$
Note that, according to \eqref{zxi}, the function $z(\cdot)$ is
continuously differentiable and there is some $\zeta \in
[0,\varrho_0/2\kappa]$ such that $\min\{z(\xi),\;\xi \in
[0,\varrho_0/2\kappa]\}=z(\zeta).$ In particular $z'(\zeta)=0$ and
one checks, thanks to \eqref{zxi}, that
$$z'(\xi)=\frac{1}{2}\sqrt{z^2+\xi^2}\exp(az^2)-\xi, \qquad \forall \xi \geq 0.$$
In particular, $z'(\zeta)=0$ is equivalent to
\begin{equation}\label{zeta}
4\zeta^2 \exp(-a z^2(\zeta))=z^2(\zeta)+\zeta^2,
\end{equation}
and one sees that $z^2(\zeta)=0$ should imply $\zeta=0$  whereas,
according to \eqref{zxi}, $z(0) \neq 0$. Consequently, all the local
extrema of $z $ are positive. Therefore,
$$\varrho_1= z(\zeta) =\min\{ z(\xi) ,\;0 \leq \xi \leq
\varrho_0/2\kappa\} >0$$ which achieves to prove that
\eqref{estimate1} holds true for some $\varrho_1 >0$. It remains now
to provide some estimate for $\varrho_1.$ Precisely, defining
$$\eta=\dfrac{1}{\sqrt{a}}\mathrm{erf}^{-1}\left(\frac{1}{2}\right),$$
 we see from \eqref{zxi} that $z^2(\xi)+\xi^2 \geq \eta^2,$ for any $\xi \geq
 0$, so that $\varrho_1^2 \geq
 \eta^2-\frac{\varrho_0^2}{4\kappa^2}$ for any $\varrho_0 \in (0,
 2\kappa \eta)$, which achieves to prove the lemma.
\end{proof}

\begin{nbA}
According to the above Lemma,  with the choice of
$\varrho_0=\frac{2\kappa \eta}{\sqrt{5}}$, one obtains that
 $\min\left(\frac{\varrho_0}{2\kappa},\frac{\varrho_1}{2}\right) \geq
\frac{\eta}{\sqrt{5}}.$  \end{nbA}


\end{document}